\documentclass[11pt]{article}

\usepackage{dashrule,arydshln,empheq,amsfonts,graphicx,amsmath,amssymb,amsthm,geometry,bbm,hyperref,nicefrac,enumerate,comment,fontawesome,mathtools,cases,stmaryrd,pict2e,picture,mathrsfs,mathabx,setspace,mathrsfs}
\usepackage{float}            
\usepackage{algorithm}
\usepackage{algpseudocode}
\usepackage{booktabs}
\usepackage{siunitx}

\usepackage{color, colortbl}
\definecolor{Gray}{gray}{0.9}

\newcommand{\Z}{\mathcal{Z}}
\newcommand{\Y}{\mathcal{Y}}

\newcommand{\R}{\mathbb{R}}
\newcommand{\N}{\mathbb{N}}
\newcommand{\E}{\mathbb{E}}

\renewcommand{\P}{\mathbb{P}}

\newcommand{\F}{\mathcal{F}}

\newcommand{\e}{\text{e}}
\renewcommand{\d}{\text{d}}

\usepackage{array}
\newcolumntype{P}[1]{>{\centering\arraybackslash}p{#1}}
\newcommand{\triple}{{\vert\kern-0.25ex\vert\kern-0.25ex\vert}}
\usepackage{mathtools}
\newcommand{\defeq}{\vcentcolon=}

\usepackage{geometry}
\usepackage{xcolor}

\newtheorem{lemma}{Lemma}[section]
\newtheorem{remark}{Remark}[section]

\newtheorem{assumption}{Assumption}

\newtheorem{theorem}{Theorem}[section]

\newtheorem{example}{Example}[section]

\theoremstyle{definition}

\newcommand{\footremember}[2]{%
    \footnote{#2}
    \newcounter{#1}
    \setcounter{#1}{\value{footnote}}%
}

\begin{document}
\title{A deep solver for backward stochastic Volterra integral equations}
\author{
  Alessandro Gnoatto\footremember{AG}{Department of Economics, University of Verona, via Cantarane, 24 - 37129 Verona, Italy.
  \\Email: \href{mailto:alessandro.gnoatto@univr.it}{alessandro.gnoatto@univr.it}}%
  \and 
  Camilo Andrés García Trillos\footremember{CT}{Department of Mathematics, University College London, Gower Street, London.
  Email: \href{mailto:camilo.garcia@ucl.ac.uk}{camilo.garcia@ucl.ac.uk}}%
  \and
    Kristoffer Andersson\footremember{KA}{ Department of Economics, University of Verona, via Cantarane, 24 - 37129 Verona, Italy.
  \\Email: \href{mailto:kristoffer.andersson@univr.it}{kristofferherbert.andersson@univr.it}}%
  \and
}

\maketitle
\vspace{0.4cm}

\begin{abstract} 
We present the first deep-learning solver for backward stochastic Volterra integral equations (BSVIEs) and their fully-coupled forward-backward variants. The method trains a neural network to approximate the two solution fields in a single stage, avoiding the use of nested time-stepping cycles that limit classical algorithms. For the decoupled case we prove a non-asymptotic error bound composed of an a posteriori residual plus the familiar square root dependence on the time step. Numerical experiments are consistent with this rate and reveal two key properties: \emph{scalability}, in the sense that accuracy remains stable from low dimension up to 500 spatial variables while GPU batching keeps wall-clock time nearly constant; and \emph{generality}, since the same method handles coupled systems whose forward dynamics depend on the backward solution. These results open practical access to a family of high-dimensional, time-inconsistent problems in stochastic control and quantitative finance.
\end{abstract}

\section{Introduction}
Backward stochastic Volterra integral equations (BSVIEs) represent a natural extension of backward stochastic differential equations (BSDEs) by allowing for memory effects and more general dependence structures. This makes them well-suited for problems in finance, control theory, and other applications where past states influence future evolution.

To motivate the need for BSVIEs, first recall a standard situation.
Assume the state process satisfies the forward stochastic differential equation (FSDE)
\begin{equation*}
    X_t
 = x_0 + \int_{0}^{t} b\bigl(s,X_s\bigr)\,\d s
      + \int_{0}^{t} \sigma\bigl(s,X_s\bigr)\,\d W_s.
\end{equation*}
For functions $g$ and $f$ that depend on arguments taken at a single time, the quantity  
\begin{equation*}
    Y_t = \E\!\left[\,g(X_T)+\int_{t}^{T} f(s,X_s)\,ds \,\big|\,\mathcal F_t\right]
\end{equation*}
admits the well known BSDE representation  
\begin{equation*}
    Y_t
 = g(X_T)
   + \int_{t}^{T} f(s,X_s)\,ds
   - \int_{t}^{T} Z_s\,dW_s.
\end{equation*}
We now expand the setting by allowing $g$ and $f$ to also depend on the evaluation date $t$.  
Set  
\begin{equation}\label{eq:cond_exp_ts}
    Y_t
 = \E\!\left[\,g(t,X_T)+\int_{t}^{T} f(t,s,X_s)\,ds \,\big|\,\mathcal F_t\right].
\end{equation}
When this extra time argument can be factorized in a multiplicative way, a change of variables restores the ordinary BSDE framework.  
For illustration, let us consider an example that can be interpreted as a pricing model in finance.  Take a possibly stochastic short rate $r$ and define  
\begin{equation*}
    g(t,x)=\exp\!\Bigl(-\!\int_{t}^{T} r_s\,ds\Bigr)\,\widetilde g(x),
\quad
f(t,s,x)=\exp\!\Bigl(-\!\int_{t}^{s} r_u\,du\Bigr)\,\widetilde f(s,x) .
\end{equation*}
Introduce the discount factor \(M_t=\exp\!\bigl(-\!\int_{0}^{t} r_u\,du\bigr)\).  
Since \(e^{-\int_{t}^{s} r_u\,du} = M_s/M_t\) we obtain  
\begin{equation*}
    Y_t\,M_t
 = \E\!\left[\,M_T\,\widetilde g(X_T)
     + \int_{t}^{T} M_s\,\widetilde f(s,X_s)\,ds \,\big|\,\mathcal F_t\right].
\end{equation*}
The term inside the expectation no longer depends on the evaluation date.  
Hence $Y_t\,M_t$ satisfies a classical BSDE and dividing by $M_t$ recovers $Y_t$.

The need for a wider Volterra formulation appears once the second time argument cannot be removed by such a factorization.  
In this scenario, the backward representation of~\eqref{eq:cond_exp_ts} is instead a BSVIE of the form
\begin{equation*}
Y_t=g(t,X_T)+\int_t^Tf(t,s,X_s)\d s-\int_t^TZ_{t,s}\d W_s.
\end{equation*}
Contrary to the BSDE case the control process $Z$ now carries two time indices, where the additional dependence on the evaluation date $t$ reflects the memory that the Volterra structure introduces and is discussed in detail in the following sections. Similar to the classical BSDE framework, the functions $f$ and $g$ may also depend on the solution pair $(Y,Z)$ itself, so that the conditional expectation in~\eqref{eq:cond_exp_ts} becomes implicit.  
In addition one can consider fully coupled systems in which the coefficients $b$ and $\sigma$ of the forward equation depend on the BSVIE solution.  

Below we present two practical situations in which the conditional expectation naturally involves both the observation time $t$ and the integration time $s$.  
In both cases, the two–clock dependence cannot be disentangled. Consequently, the standard BSDE machinery breaks down, making a Volterra formulation essential.
\begin{example}[Social discounting, see \cite{brody2018social}]
    A sovereign wealth fund must value uncertain long-dated cash-flows $C_s$ over the horizon $[0,T]$.  
    To reflect inter-generational welfare it adopts the declining hyperbolic kernel
    \[
        D(t,s)=\frac{1}{\bigl(1+\alpha(s-t)\bigr)^{\beta}},\qquad s\ge t,\ \alpha,\beta>0,
    \]
    whose slower decay ensures that payments decades ahead are not virtually ignored.  
    The present social cost at time $t$ is
    \[
        Y_t=\E\!\left[\int_t^T D(t,s)\,C_s\,\emph{d}s\;\middle|\;\F_t\right].
    \]
    Because $D$ depends on both the evaluation time $t$ and the integration time $s$, no driver of the form $f(s,\cdot)$ exists, so a BSDE representation fails.  
    The appropriate formulation is the BSVIE
    \[
        Y_t=\int_t^T D(t,s)\,C_s\,\emph{d}s-\int_t^T Z_{t,s}\,\emph{d}W_s,
    \]
    whose control field $Z_{t,s}$ carries both time indices and captures the memory effect of the moving discount kernel.
\end{example}

\begin{example}
Suppose a bank enters into an OTC derivative position with terminal payoff $g$ against a counterparty subject to default risk. The default risk of the counterparty is described, \textcolor{black}{ in a standard reduced-form setting for credit risk, }by the deterministic positive hazard rate function $\lambda$. We assume that, in case of default, the reference value of the transaction is the whole value of the position including the valuation adjustment due to counterparty risk. This gives rise to the valuation formula (see e.g. \cite{biagini2021unified}, \textcolor{black}{for a general treatment}) 
\begin{align}
\label{eq:fullFBSDE}
    Y_t=\mathbb{E}\left[\left.e^{-\int_t^T r_s \emph{d}s}g(X_T)-LGD\int_t^Te^{-\int_t^s (r_u+\lambda_u) \emph{d}u}(Y_s)^+\textcolor{black}{\lambda_s}\emph{d} s\right|\mathcal{F}_t\right]
\end{align}
where $r$ is the deterministic function describing the overnight rate and $LGD\in(0,1]$ is a constant loss given default. \textcolor{black}{The first term in \eqref{eq:fullFBSDE} represents the clean value, in the terminology of \cite{biagini2021unified}, i.e. the standard risk-neutral price of the claim, whereas the integral term gives rise to the credit value adjustment.} While the discount factor can be factorized as above, the presence of the hazard rate turns this valuation problem into a BSVIE
\begin{align}
Y_t = e^{-\int_t^T r_s \emph{d}s}g(X_T)-LGD\int_t^Te^{-\int_t^s (r_u+\lambda_u) \emph{d}u}(Y_s)^+\textcolor{black}{\lambda_s}\emph{d} s-\int_t^T Z_{t,s}\emph{d} W_s    
\end{align}
\end{example}

Beyond serving as backward representations of conditional expectations, BSVIEs provide a powerful framework for time inconsistent stochastic optimal control problems. When an objective functional changes with the initial date, whether through non exponential discounting, dynamic risk measures, or any other mechanism, the dynamic programming principle collapses and the classical BSDE adjoint is no longer adequate. Because a BSVIE driver may depend simultaneously on the evaluation time and the integration time, it naturally captures the evolving preferences and path dependence that come with time inconsistency, yielding equilibrium conditions that can be derived either through a Pontryagin maximum principle or, in an HJB-like formulation, by treating the BSVIE variable $Y$ as the value function. In practice, the BSVIE is coupled with an FSDE or a forward stochastic Volterra integral equation (FSVIE), producing a well-posed FSDE-BSVIE or forward-backward stochastic Volterra integral equation (FBSVIE) system whose solution characterizes admissible equilibrium controls even when the underlying value process is non-Markovian. See, for instance, \cite{yong2012time,wang2021time,wang2021closed} for time-inconsistent stochastic optimal control problems formulated as FBSVIEs. In \cite{wang2024optimal} the authors study optimal control of FBSDEs, where the state itself follows a controlled FBSDE. This setting naturally leads to an FBSDE-BSVIE system and even accommodates a conditional mean variance portfolio optimization problem.

Finally, we mention dynamic risk measures, which is another area where the two-time structure of a BSVIE is essential. When the exposure is a whole cash-flow stream and the risk weights or discount factors shift with the observation date, a classical BSDE cannot retain the resulting memory and horizon-dependence, whereas a BSVIE captures both through its simultaneous $(t,s)$ driver. This representation underlies modern insurance reserving, capital allocation rules for banking groups, and portfolio policies constrained by pathwise VaR or CVaR limits; see, for example, \cite{yong2007continuous,di2024fully,wang2021recursive}. 

Since the seminal work of \cite{lin2002adapted} launched the field, the theory of BSVIEs has expanded rapidly, see, for instance, the introduction of type--II BSVIEs (when the driver takes both $Z_{t,s}$ and $Z_{s,t}$ as inputs) in \cite{yong2006backward}, the well‑posedness and regularity of M‑solutions for type--II BSVIEs\footnote{\color{black}Type-I BSVIEs are those where $Y(t)$ depends only on $Z(t,s)$ for $s \ge t$, while type-II BSVIEs may also involve $Z(s,t)$, see e.g., \cite{yong2013backward} for a clear survey.\color{black}} in \cite{yong2008well} and control–theoretic applications surveyed in \cite{yong2013backward}.  By contrast, the numerical side has hardly kept pace.

On the numerical side, the literature is remarkably thin: to the best of our knowledge there are only three published schemes for BSVIEs: (i) the finite‑difference analysis in the PhD thesis of \cite{pokalyuk2012discretization};  
(ii) an implicit backward‑Euler scheme for type‑I BSVIEs proposed by \cite{wang2016numerical}; and  
(iii) the recent explicit backward‑Euler method for type‑II equations of \cite{hamaguchi2023approximations}. The discretization schemes put forward in \cite{pokalyuk2012discretization,wang2016numerical,hamaguchi2023approximations} assume that the conditional expectations can be computed exactly; consequently, they remain semi-discrete and cannot be executed without an additional (unspecified) approximation layer. To the best of our knowledge, the algorithm developed in the following is the first fully implementable end-to-end solver for BSVIEs.
 
By contrast, BSDEs already have several neural-network solvers, pioneered by the seminal deep-BSDE method of Han, Jentzen, and E \cite{han2018solving}. Following this work, several deep learning-based strategies have emerged, notably \cite{beck2021solving,beck2019machine,raissi2024forward,ji2020three,andersson2022deep,andersson2025deep}, with convergence analyses provided in, e.g., \cite{han2020convergence,hutzenthaler2020proof,Grohs2018APT,jentzen2018proof,andersson2023convergence,reisinger2024posteriori, negyesi2024generalized, gnoatto2025convergence}. Concurrently, a separate branch known as backward-type methods, closer in spirit to classical dynamic programming algorithms, has developed, see e.g., \cite{hure2020deep,chan2019machine,fang2009novel,Balint,kapllani2025backward,germain2022approximation}. For a comprehensive survey of numerical methods for approximating BSDEs and PDEs, see \cite{chessari2023numerical}, and for neural-network-based approaches specifically, we refer to \cite{beck2023overview}. Yet no neural‑network‑based method has been put forward for BSVIEs. This imbalance between abundant theory and scarce algorithms motivates the present work.

In this paper, we contribute to the literature on BSVIEs in the following ways: \begin{enumerate} \item We propose a neural network-based approximation method for BSVIEs. The method directly approximates the functional form of the BSVIE solution processes and does not rely on reformulating the problem as a family of BSDEs. \item We prove that the overall (mean-squared) simulation error of our method can, up to a multiplicative constant, be bounded by the size of the temporal discretization and the mean-squared error of the free-term condition (the BSVIE analog of the terminal condition in BSDEs). In this sense, our analysis extends the results of \cite{han2020convergence} to the BSVIE framework. \end{enumerate}
Taken together, these two results deliver the first machine‑learning solver for BSVIEs that comes with a non‑asymptotic error guarantee, thereby closing the algorithmic gap identified above and paving the way for high‑dimensional, path‑dependent applications that were previously out of reach.

In Section~\ref{sec:preliminaries} we state the class of equations studied and the assumptions that guarantee existence and uniqueness of their solutions.  
Section~\ref{sec:var_discretization} reformulates the problem in variational form and introduces its Euler--Maruyama time discretization.  
The error analysis is carried out in Section~\ref{sec:error_analysis}.  
Section~\ref{sec:fully_NN} describes the complete algorithm and the main neural-network details.  
Finally, Section~\ref{sec:numerical} reports numerical results.

\section{Preliminaries}\label{sec:preliminaries}
This section sets out the notation and presents the two problem settings we study. First, we introduce the decoupled FSDE-BSVIE and list the assumptions that ensure a unique adapted solution. This is the setting used for the error analysis. We then show how the same algorithm applies to a coupled FSDE-BSVIE and briefly recall the known existence and uniqueness results for this case. The algorithm can be implemented for both settings, but the error analysis in this paper is carried out only for the first one.

Throughout this paper, we let $T\in(0,\infty)$, $d,\ell\in\N$, $x_0\in\R^d$, $(W_t)_{t\in[0,T]}$ be an $\ell-$dimensional standard Brownian motion on a filtered probability space $(\Omega,\F,(\F_t)_{t\in[0,T]},\P)$ and define $\Delta[0,T]^2\defeq\{(t,s)\in[0,T]^2\,\big|\,t\leq s\}$.

\subsection{A decoupled FSDE-BSVIE}\label{sec:decoupled_FBSVIE}
In this subsection, the problem coefficients are given by $x_0\in\R^d$, $b\colon [0,T]\times\R^d\to\R^d$, $\sigma\colon[0,T]\times\R^d\to\R^{d\times \ell}$, $g\colon[0,T]\times\R^d\to\R$ and $f\colon\Delta[0,T]^2\times\R^d\times\R\times\R^\ell\to\R$. We consider a decoupled FSDE-BSVIE of the form
\begin{equation}
\begin{cases}
    \label{eq:FBSVIE}
    \displaystyle
    X_t=x_0 + \int_0^tb(s,X_s)\d s + \int_0^t\sigma(s,X_s)\d W_s,
    \\[6pt]
    \displaystyle Y_t=g(t,X_T) + \int_t^Tf(t,s,X_s,Y_s,Z_{t,s})\d s -\int_t^TZ_{t,s}^\top\d W_s.
    \end{cases}
\end{equation}
 We let the following assumptions, which are equivalent to Assumptions A1-A4 in \cite{wang2016numerical} hold true.
\begin{assumption}\label{ass:1}
For $(t_1,s_1),(t_2,s_2)\in\Delta [0,T]^2$, $x\in \R^d$, $y\in\R$ and $z\in\R^\ell$, there exists a constant $K_1$ such that
\begin{align*}
    |f(t_1,s_1,x,y,z)-f(t_2,s_2,x,y,z)| + |g(t_1,x)-g(t_2,x)|&\leq K_1(|t_1-t_2|^{1/2} + |s_1-s_2|^{1/2}),\\
    |f(\cdot,\cdot,0,0,0)| + |g(\cdot,0)|&\leq K_1.
\end{align*}
\color{black}
Moreover, \(f\) and \(g\) have continuous and uniformly bounded first- and second-order
partial derivatives with respect to \(x,y,z\) (for \(f\)) and \(x\) (for \(g\)), with all these
derivatives bounded by \(K_1\). In particular, \(f\) is globally Lipschitz in \((x,y,z)\) with Lipschitz constant at most \(K_1\),
and \(g\) is globally Lipschitz in \(x\) with Lipschitz constant at most \(K_1\).
\color{black}
\end{assumption}

\begin{assumption}\label{ass:2}
For $x,y\in \R^d$, and $t,s\in[0,T]$, there exists a constant $K_2$ such that
\begin{align*}
    |b(t,x)-b(t,y)|+|\sigma(t,x)-\sigma(t,y)|&\leq K_2|x-y|,\\
    |b(t,x)-b(s,x)| + |\sigma(t,x)-\sigma(s,x)|&\leq K_2|t-s|^{1/2},\\
    |b(\cdot,0)| + |\sigma(\cdot,0)|&\leq K_2.
\end{align*}
\end{assumption}

The above assumptions are sufficient to guarantee the existence of a unique adapted solution to the decoupled FSDE-BSVIE~\eqref{eq:FBSVIE}, which is stated in the following theorem.

\begin{theorem}\label{thm:wang_etc}
    Under assumptions~\ref{ass:1}-\ref{ass:2}, it holds that:\begin{enumerate}
        \item The FSDE~\eqref{eq:FBSVIE} admits a unique adapted, continuous, square-integrable, solution $X=(X_t)_{t\in[0,T]}$,
        \item The BSVIE~admits a unique solution $(Y,Z)=(Y_s,Z_{t,s})_{(t,s)\in\Delta[0,T]^2)}$  where $Y$ is adapted and, for each fixed $t$, the map $s\mapsto Z_{t,s}$ is $\mathcal F_s$‑adapted on $[t,T]$.
    \end{enumerate}
    \begin{proof}
    For 1, we refer to classical results on well-posedness for FSDEs in \cite{oksendal2013stochastic} and 2 is exactly \cite[Theorem 2.3]{wang2016numerical}.
    \end{proof}
\end{theorem}

\subsection{A coupled FSDE-BSVIE}
In this subsection, the problem coefficients are given by $x_0\in\R^d$, $b\colon [0,T]\times\R^d\times\R\times\R^\ell\to\R^d$,  $\sigma\colon[0,T]\times\R^d\times\R\to\R^{d\times \ell}$, $g\colon[0,T]\times\R^d\times\R^d\to\R$ and $f\colon\Delta[0,T]^2\times\R^d\times\R\times\R^\ell\times\R^\ell\to\R$. We consider a coupled FSDE-BSVIE of the form
\begin{equation}
\begin{cases}
    \label{eq:coupled_FSDE_BSVIE}
    \displaystyle
    X_t=x_0 + \int_0^tb(s,X_s,Y_s,Z_{s,s})\d s + \int_0^t\sigma(s,X_s,Y_s)\d W_s,
    \\[6pt]
    \displaystyle Y_t=g(t,X_t,X_T) + \int_t^Tf(t,s,X_s,Y_s,Z_{t,s},Z_{s,s})\d s -\int_t^TZ_{t,s}^\top\d W_s.
    \end{cases}
\end{equation}
The coupled FSDE-BSVIE was first presented in~\cite{wang2021time} as a model for a time-inconsistent stochastic optimal control problem. In that work the authors showed that \eqref{eq:coupled_FSDE_BSVIE} can be rewritten as an HJB equation. They proved that, under suitable conditions, the HJB has a classical solution, which in turn guarantees that the FSDE-BSVIE has a unique adapted solution. Because we do not study this example in our error analysis, we omit the detailed conditions and refer the reader to~\cite{wang2021time} for full descriptions.

\section{Variational formulations and temporal discretization}\label{sec:var_discretization}
We first rewrite the problem as a variational problem that is continuous in time. This form is similar to the one used in the deep BSDE method~\cite{han2016deep}, but is adjusted to BSVIEs. Next, we discretize the time interval to obtain a semi-discrete problem, which then serves as the starting point of our numerical method.
\subsection{A time continuous variational formulation}
We begin with the familiar variational formulation for a standard FBSDE. This example gives the reader a clear reference point before we extend the same ideas to the broader FSDE-BSVIE cases that follow. Consider the following variational problem:
\begin{equation}\label{eq:var_FBDE}\begin{dcases}
\underset{y_0,\Z}{\mathrm{minimize}}\ 
 \E|Y_T^{y_0,\mathcal{Z}}-g(X_T^{y_0,\mathcal{Z}})|^2,\quad \text{where for}\quad t\in[0,T],\\
X_t^{y_0,\Z}=x_0+\int_0^tb(s,X_s^{y_0,\Z},Y_s^{y_0,\Z},\Z_s)\d s + \int_0^t\sigma(s,X_s^{y_0,\Z},Y_s^{y_0,\Z},\Z_s)\d W_s,\\
Y_t^{y_0,\mathcal{Z}}=y_0 - \int_0^tf(s,X_s^{y_0,\Z},Y_s^{y_0,\Z},\Z_s)\d s + \int_0^t(Z_s^{y_0,\Z})^\top\d W_s.
    \end{dcases}
\end{equation}
A solution to the FBSDE
\begin{equation}
\begin{cases}
    \label{eq:coupled_FBSDE}
    \displaystyle
    X_t=x_0 + \int_0^tb(s,X_s,Y_s,Z_s)\d s + \int_0^t\sigma(s,X_s,Y_s,Z_s)\d W_s,
    \\[6pt]
    \displaystyle Y_t=g(X_T) + \int_t^Tf(s,X_s,Y_s,Z_s)\d s -\int_t^TZ_s^\top\d W_s,
    \end{cases}
\end{equation}
clearly solves \eqref{eq:var_FBDE} and standard well-posedness conditions for the FBSDE guarantee that this solution is unique. Moreover, under suitable conditions, one further obtains the feedback forms $y_0=y_0(x_0)$ and $\Z_t = \Z\bigl(t, X_t^{y_0,\Z}\bigr)$. This variational formulation is exactly what inspires the deep-BSDE method~\cite{han2018solving}.

Just as an FBSDE can be written in a variational form, an FSDE-BSVIE can as well.
We seek processes \(Y\) and \(Z\) that satisfy the free-term dynamics.
Choose coefficients \(\varphi, b, \sigma, g,\) and \(f\) such that the FSDE-BSVIE system is either decoupled, as  in~\eqref{eq:FBSVIE} or coupled, as in~\eqref{eq:coupled_FSDE_BSVIE}.
We assume that the chosen system has a unique adapted solution.
With this assumption the equation is equivalent to the following variational problem:
\begin{equation}\label{eq:var_FBSVIE}\begin{dcases}
\underset{\Y,\Z}{\mathrm{minimize}}\ \int_0^T
 \E|\mathscr{Y}_T^{\mathcal{Y},\mathcal{Z}}(t)-g(t,X_t^{\mathcal{Y},\mathcal{Z}},X_T^{\mathcal{Y},\mathcal{Z}})|^2\d t,\quad \text{where for}\quad t\in[0,T],\\
X_t^{\Y,\Z}=x_0+\int_0^tb(s,X_s^{\Y,\Z},\Y_s,\Z_{s,s})\d s + \int_0^t\sigma(s,X_s^{\Y,\Z},\Y_s)\d W_s,\\
\mathscr{Y}_T^{\mathcal{Y},\mathcal{Z}}(t)=\Y_t - \int_t^Tf(t,s,X_s^{\Y,\Z},\Y_s,\Z_{t,s},\Z_{s,s})\d s + \int_t^T\Z_{t,s}^\top\d W_s.
    \end{dcases}
\end{equation}
Here $\mathscr{Y}_T^{\Y,\Z}(t)$ is $\F_T$-measurable and can be viewed as the target from the perspective of the evaluation time $t$. Concretely, at time $t$ one wants to find processes $\bigl(\Y_s\bigr)_{s\in[t,T]}$ and $\bigl(\Z_{t,s}\bigr)_{s\in[t,T]}$ on the interval $[t,T]$ such that
\begin{equation*}
  \mathscr{Y}_T^{\Y,\Z}(t) =
g\bigl(t,X_t^{\Y,\Z},X_T^{\Y,\Z}\bigr).
\end{equation*}
Then, at a later time $t+u\leq T$, one looks again for processes $\bigl(\Y_s\bigr)_{s\in[t+u,T]}$ and $\bigl(\Z_{t+u,s}\bigr)_{s\in[t+u,T]}$ so that
\begin{equation*}
\mathscr{Y}_T^{\Y,\Z}(t+u) \;=\;
g\bigl(t+u,\;X_{t+u}^{\Y,\Z},\;X_T^{\Y,\Z}\bigr).
\end{equation*}
Importantly, this should not be interpreted to mean that $\Y_s$ and $\Z_{t,s}$ for $s \in [t,T]$ must be $\mathcal{F}_t$-measurable. Instead, it simply reflects the idea that, at each time $t$, one considers the relevant processes over the time interval $[t,T]$ in order to satisfy the above relationship for $\mathscr{Y}_T^{\Y,\Z}(t)$.

In this setting, under suitable conditions, the processes \(Y_\cdot\) and \(Z_{\cdot,\cdot}\) can be 
expressed in the feedback forms $\Y_t = \Y\bigl(t, X_t^{\Y,\Z}\bigr)$, and  $\Z_{t,s} = \Z\bigl(t,s, X_t^{\Y,\Z}, X_s^{\Y,\Z}\bigr)$.

A solution to \eqref{eq:coupled_FSDE_BSVIE} clearly solves \eqref{eq:var_FBSVIE} since the objective 
function becomes identically zero. Furthermore, under appropriate conditions ensuring 
well-posedness of \eqref{eq:coupled_FSDE_BSVIE}, this solution is unique. 

\subsection{A time discrete variational formulation}

For some $N\in\N$, let $\pi\coloneq \{0=t_0<t_1<\cdots<t_{N-1}<t_N=T\}$ be an equidistant grid with $h=T/N$ and define $\pi(t)\defeq\inf\{k\,\big|\,t\in[t_k,t_{k+1})\}$ and $\Delta^\pi[0,T]^2\defeq\{\big(\pi(t),\pi(s)\big)\,\big|\,(t,s)\in\Delta [0,T]^2\}$. For $n\in\{0,1,\ldots,N-1\}$, let $\Delta W_n$ be a vector of $\ell$ i.i.d. normally distributed random variables, each with mean 0 and variance $h$. 

We now present the semi-discrete formulation of the variational problem \eqref{eq:var_FBSVIE}, under the assumption that the functions $\varphi$, $b$, $\sigma$, $f$, and $g$ are chosen such that the system corresponds to either the form given in \eqref{eq:FBSVIE} or in \eqref{eq:coupled_FSDE_BSVIE}.
This formulation employs the feedback forms for \(Y\) and \(Z\) introduced above, and reads
\begin{equation}\label{eq:var_disc_FBSVIE}\begin{dcases}
\underset{\Y,\Z}{\mathrm{minimize}}\ \sum_{n=0}^{N-1}
 \E\big|\mathscr{Y}_N^{\pi,\mathcal{Y},\mathcal{Z}}(n)-g\big(t_n,X_n^{\pi,\mathcal{Y},\mathcal{Z}},X_N^{\pi,\mathcal{Y}}\big)\big|^2h,\quad \text{where for}\quad t\in[0,T],\\
X_n^{\Y,\Z}=x_0+\sum_{k=0}^{n-1}b\big(t_k,X_k^{\pi,\Y,\Z},Y_k^{\pi,\Y,\Z}, Z_{k,k}^{\pi,\Y,\Z}\big)h
\\
\qquad\qquad+ \sum_{k=0}^{n-1}\sigma\big(t_k,X_k^{\pi,\Y,\Z},Y_k^{\pi,\Y,\Z}\big)\Delta W_k,\\
\mathscr{Y}_N^{\pi,\mathcal{Y},\mathcal{Z}}(n)=Y_n^{\pi,\Y,Z} - \sum_{k=n}^{N-1}f(t_n,t_k,X_k^{\pi,\Y,\Z},Y_k^{\pi,\Y,\Z},Z_{n,k}^{\pi,\Y,\Z},Z_{k,k}^{\pi,\Y,\Z})h + \sum_{k=n}^{N-1}(Z_{n,k}^{\pi,\Y,\Z})^\top\Delta W_k,\\
Y_n^{\pi,\Y,\Z}=\Y(t_n,X_n^{\pi,\Y,\Z}), \quad Z_{n,k}^{\pi,\Y,\Z}=\Z(t_n,t_k,X_n^{\pi,\Y,\Z},X_k^{\pi,\Y,\Z}).
    \end{dcases}
\end{equation}
Note that when we consider a decoupled FSDE-BSVIE, we have assumed that $g(t,x_t,x)=g(t,x)$, and hence the feedback form for $\Z$ reduces to $\Z(t_n,t_k,X_k^{\pi,\Y,\Z})$.

At this point the scheme is almost ready to implement.  
The final aspects to be cleared are
\begin{enumerate}
  \item an approximation of the expectation,  
  \item explicit prescriptions for the functions $\mathcal{Y}$ and $\mathcal{Z}$, and  
  \item a practical optimization routine.
\end{enumerate}
\paragraph*{Expectation.}
\paragraph*{Expectation.}
We approximate the expectations appearing in the loss by Monte Carlo simulation. 
For any fixed parameter choice, the Monte Carlo estimator of the expected loss is unbiased, with mean--square error $O(M^{-1})$ (equivalently, standard error $O(M^{-1/2})$) as the number of samples $M$ increases. 
Note that this rate concerns the accuracy of the Monte Carlo approximation of the expectation operator, not the convergence of the loss itself. 
The loss function is already a mean--squared quantity by definition, what converges with rate $O(M^{-1})$ is the estimator of that expectation, not the loss or its optimizer.

\paragraph*{Functions $\mathcal{Y}$ and $\mathcal{Z}$.}
In a subsequent section we represent $\mathcal{Y}$ and $\mathcal{Z}$ with neural networks. For the error analysis in the next section, however, we keep the choice of function class completely open.
\paragraph*{Optimization.}
The numerical experiments later in the paper employ mini–batch stochastic gradient descent with the Adam algorithm.  
The specific optimizer does not influence the error bounds developed in the following section, so its details are omitted for now.

\section{Error analysis}\label{sec:error_analysis}
In this section, we assume the setting of Section~\ref{sec:decoupled_FBSVIE}, i.e., we let the randomness of the BSVIE stem from an FSDE, which does not take the solution of a backward equation as inputs
\begin{equation}\label{eq:FSVIE}
        X_t=x_0 + \int_0^tb(s,X_s)\d s + \int_0^t\sigma(s,X_s)\d W_s.
\end{equation}
The aim is to prove an a posteriori error bound for the BSVIE
\begin{equation}
    \label{eq:BSVIE_reduced}
    \displaystyle Y_t=g(t,X_T) + \int_t^Tf(t,s,X_s,Y_s,Z_{t,s})\d s -\int_t^TZ_{t,s}^\top\d W_s.
\end{equation} We work under Assumptions~\ref{ass:1}-\ref{ass:2}; by Theorem~\ref{thm:wang_etc} these conditions guarantee a unique adapted solution.

Below we introduce the Euler--Maruyama scheme for the FSDE in~\eqref{eq:FBSVIE} which is used throughout this section
\begin{align} \label{eq:FSVIE_disc}
                \displaystyle X_n^\pi= x_0 + \sum_{k=0}^{n-1}b(t_k,X_k^\pi)h + \sum_{k=0}^{n-1}\sigma(t_k,X_k^\pi)\Delta W_k,\quad n\in\{0,1,\ldots,N\},
\end{align}
where the empty sum should be interpreted as zero. The following theorem states convergence and square integrability of \eqref{eq:FSVIE_disc}.
\begin{theorem}
    Let Assumption~\ref{ass:2} hold true. Then there exist constants $C_x$ and $K_x$ such that
    \begin{equation*}
        \sup_{t\in[0,T]}\E|X_t-X_{\pi(t)}^\pi|^2\leq C_x h,\quad \max_{n\in\{0,1,\ldots N\}}\E|X_n^\pi|^2\leq K_x.
    \end{equation*}
    \begin{proof}
        See for instance~\cite{kloedenplaten1992}.
    \end{proof}
\end{theorem} 
Below we present a generic discretization scheme for a BSVIE of the form~\eqref{eq:BSVIE_reduced}
\begin{equation}\label{eq:BSVIE_disc}\begin{cases}
        \displaystyle
    \mathscr{Y}_N^\pi(n)=Y_n^\pi - \sum_{k=n}^{N-1}f(t_n,t_k,X_k^\pi,Y_k^\pi,Z_{n,k}^\pi)h+\sum_{k=n}^{N-1}(Z_{n,k}^\pi)^\top\Delta W_k,
    \\[6pt]
    Y_n^\pi=\Y(t_n,X_n^\pi),\quad Z_{n,k}^\pi=\Z(t_n,t_k,X_k^\pi),\quad (n,k)\in\Delta^\pi[0,T]^2.
    \end{cases}
\end{equation}
The scheme above is of feedback form for the approximations of $Y$ and $Z$. Moreover, it uses an Euler--Maruyama scheme to compute $\mathscr{Y}_N^\pi(n)$. 

\begin{assumption}\label{ass:4}
    The functions $\Y\colon[0,T]\times\R^d$ and $\Z\colon\Delta[0,T]^2\times\R^d\to\R^\ell$ satisfy a linear growth condition, i.e., for $t_0\in[0,T]$, $(t,s)\in\Delta[0,T]^2$, $x\in\R^d$ there exist constants $K_\Y$ and $K_\Z$ such that \begin{align*}
        |\Y(t_0,x)|^2\leq K_\Y(1+|x|^2),\quad |\Z(t,s,x)|^2\leq K_\Z(1+|x|^2).
    \end{align*}
\end{assumption}

In the following, we present the main result of this section, which is an a posteriori error estimate for the above scheme.
\begin{theorem}\label{thm:main_result}
Let Assumptions~\ref{ass:1}-\ref{ass:4} hold true and suppose that $f=f(t,s,x,y)$ or $f=f(t,s,x,z)$. Then for sufficiently small $h$, there exists a constant $C$, depending on $T$ and $K_1$, such that
\begin{equation*}
    \int_0^T\E|Y_t-Y_{\pi(t)}^\pi|^2\emph{d} t + \int_0^T\int_t^T\E|Z_{t,s}-Z_{\pi(t),\pi(s)}^\pi|^2\emph{d} s\emph{d} t\leq C\big(h+\sum_{n=0}^{N-1}\E|\mathscr{Y}_N^\pi(n)-g(t_n,X_N^\pi)|^2h\big).
\end{equation*}
\end{theorem}
\noindent
To prove Theorem~\ref{thm:main_result}, we first present several intermediate results which are used in the final argument. The overall proof strategy is as follows:
\begin{enumerate}
    \item \textbf{Approximation by coupled BSDEs.}
    We introduce a family of $N$ coupled BSDEs that approximates the BSVIE~\eqref{eq:BSVIE_reduced} in the limit as $h \to 0$ (equivalently $N \to \infty$).

    \item \textbf{Discretization of the BSDE family.}
    We present an explicit backward Euler--Maruyama scheme for these BSDEs which converges to the solution of the BSDE family.

    \item \textbf{Stability estimate.}
    We derive a stability result that compares (i) the BSVIE scheme \eqref{eq:BSVIE_disc} and (ii) the discretized BSDE family \eqref{eq:BSDE_family_disc}. 
    This gives a precise bound on the difference between the two schemes.

    \item \textbf{Conclusion of the proof.}
    Finally, we combine the convergence of the BSDE scheme and the stability estimate to conclude that our BSVIE scheme converges, thereby establishing the error bounds in Theorem~\ref{thm:main_result}.
\end{enumerate}
We note that Steps~1--2 in our outline 
are established directly via the main results of \cite{wang2016numerical}.  
Meanwhile, Steps~3--4, are in spirit similar to the a posteriori error bounds for coupled FBSDE presented in \cite{han2020convergence}.  

In the remainder of this section, we introduce the coupled BSDEs in \eqref{eq:BSDE_family}, recall key results from \cite{wang2016numerical} regarding their numerical discretization, and prove the necessary stability estimates.  At the end, we assemble all of these ingredients in a final proof of Theorem~\ref{thm:main_result}.

To approximate the BSVIE \eqref{eq:BSVIE_reduced}, 
we introduce a family of $N$ coupled BSDEs, one for each $k=0,\ldots,N-1$.  
In the limit as $N \to \infty$ (equivalently $h \to 0$), 
these BSDEs collectively recover the BSVIE solution.  
Concretely, for $0\leq k\leq N-1$ and $t\in[t_k,T]$, we set
\begin{equation}\label{eq:BSDE_family}
    Y^k_t=g(t_k,X_T) + \int_t^Tf(t_k,s,X_s,Y_s^{\pi(s)},Z_s^k)\d s -\int_t^TZ_s^k\d W_s.
\end{equation}
Note that this defines a system of $N$ coupled BSDEs, where $Y^k$ is defined for $t\in[t_k,T]$. For each $n\in \{k+1,\ldots,N-1\}$, the driver of the $k$:th BSDE takes $Y^n$ and $Z^k$ as inputs in $f$ on the interval $[t_n,t_{n+1})$. The following Theorem states that for each $k$ \eqref{eq:BSDE_family} admits a unique adapted solution, and converges to \eqref{eq:BSVIE_reduced} as the size of the temporal steps goes to zero.
\begin{theorem}\label{thm:wang_etc_2}
    Under assumptions~\ref{ass:1}-\ref{ass:2}, it holds that:\begin{enumerate}
        \item The family of BSDEs~\eqref{eq:BSDE_family} admits a unique adapted solution $(Y_s^{\pi(t)},Z_s^{\pi(t)})_{(t,s)\in\Delta[0,T]^2}$.
\item For each $k\in\{0,1,\ldots,N-1\}$, there exists a constant $C_1$, depending on $T$ and $K_1$, such that \begin{equation*}
        \int_0^T\E\bigl|Y_t-Y_t^{\pi(t)}\bigr|^2\emph{d} t + \int_0^T\int_t^T\E\bigl|Z_{t,s}-Z_s^{\pi(t)}\bigr|^2\emph{d} s\emph{d} t\leq C_1h.
    \end{equation*}
    \end{enumerate} 
    \begin{proof}
    This Theorem states the results of Theorem 2.3 and Lemma 4.12 in \cite{wang2016numerical}.
    \end{proof}
\end{theorem}

The scheme below, proposed in \cite{wang2016numerical}, is an explicit backward type Euler--Maruyama scheme for the family of BSDEs~\eqref{eq:BSDE_family}.
\begin{equation}\label{eq:BSDE_family_disc}\begin{cases}
        \displaystyle
Y_n^{k,\pi}=\E\big[Y_{n+1}^{k,\pi}\,\big|\,\F_{t_n}\big] + f(t_k,t_n,X_n^\pi,Y_n^{n,\pi},Z_n^{k,\pi})h,\\[6pt]
    Z^{k,\pi}_n=\frac{1}{h}\E\big[\Delta W_n Y_{n+1}^{k,\pi}\,\big|\,\F_{t_n}\big],\\[6pt]
    Y_N^{k,\pi}=g(t_k,X_N^\pi), \quad  (n,k)\in\Delta^\pi[0,T]^2.
    \end{cases}
\end{equation}

The following theorem, which combines results from \cite{wang2016numerical}, states that the scheme~\eqref{eq:BSDE_family_disc} converges 
to the solution of the BSDE family \eqref{eq:BSDE_family}.
\begin{theorem}\label{thm:BSDE_disc_conv}
Let Assumption~\ref{ass:1}-\ref{ass:2} hold true and suppose that $f=f(t,s,x,y)$ or $f=f(t,s,x,z)$. Then, for each $k\in\{0,1,\ldots,N-1\}$, there exists a constant $C_2$, depending on $T$ and $K_1$, such that \begin{equation*}
    \int_{t_k}^{t_{k+1}}\E\big|Y_t^k-Y_{\pi(t)}^{k,\pi}\big|^2\emph{d} t+ h\int_{t_k}^T\E\big|Z_t^k-Z_{\pi(t)}^{k,\pi}\big|^2\emph{d} t\leq C_2h.
    \end{equation*}
    \begin{proof}
    This is a direct application of \cite[Lemma 4.5 and Lemma 4.12]{wang2016numerical}.
    \end{proof}
\end{theorem}

\begin{remark}
The scheme in \cite{wang2016numerical} uses $Y_{n+1}^{n,\pi}$ rather than $Y_n^{n,\pi}$ in the driver. This makes the scheme implicit in the equation for $\widebar{Z}_n^{k\pi}$ since $f(t_k,t_n,X_n^\pi,Y_{n+1}^{n,\pi},Z_n^{k,\pi})$ is $\F_{t_{n+1}}$ measurable (rather than $\F_{t_n}$-measurable) which yields $Z^{k,\pi}_n=\frac{1}{h}\E\big[\Delta W_n Y_{n+1}^{k,\pi}+ f(t_k,t_n,X_n^\pi,Y_{n+1}^{n,\pi},Z_n^{k,\pi})\Delta W_n\,\big|\,\F_{t_n}\big]$. Nevertheless, the same techniques apply, and thus one can prove the same error bound 
for the explicit scheme considered here. Alternatively, one can use the more general scheme for type-II BSVIEs proposed in~\cite{hamaguchi2023approximations}, which when applied to type-I BSVIEs coincides with \eqref{eq:BSDE_family_disc}.
\end{remark}

To bound the difference between our BSVIE scheme~\eqref{eq:BSVIE_disc} and the scheme for the BSDE family~\eqref{eq:BSDE_family_disc}, we introduce the following general scheme
\begin{equation}\label{eq:BSDE_family_disc_generic}
        \displaystyle
    \widehat{Y}_{n+1}^{k,\pi}=\widehat{Y}_n^{k,\pi} - f(t_k,t_n,X_n^\pi,\widehat{Y}_n^{n,\pi},\widehat{Z}_n^{k,\pi})h+(\widehat{Z}_n^{k,\pi})^\top\Delta W_n,\quad (k,n)\in\Delta^\pi[0,T]^2.
\end{equation}

Because no initial or terminal condition is imposed, the general scheme above admits infinitely many solutions. Furthermore, we observe that 
$\widehat{Y}_n^{k,\pi}=\E\big[\widehat{Y}_{n+1}^{k,\pi}\,\big|\,\F_{t_n}\big] + f(t_k,t_n,X_n^\pi,\widehat{Y}_n^{n,\pi},\widehat{Z}_n^{k,\pi})h$, and 
$\widehat{Z}^{k,\pi}_n=\frac{1}{h}\E\big[\Delta W_n \widehat{Y}_{n+1}^{k,\pi}\,\big|\,\F_{t_n}\big]$, implying that \eqref{eq:BSDE_family_disc} is a special case of \eqref{eq:BSDE_family_disc_generic}. It is, in fact, also possible to express our BSVIE scheme via the generic scheme~\eqref{eq:BSDE_family_disc_generic}. To illustrate this, we introduce the following notation
\begin{equation}\label{eq:BSVIE_BSDE_disc}
\begin{cases}
    \displaystyle
\mathscr{Y}_k^\pi(n+1)=\mathscr{Y}_k^\pi(n) - f(t_k,t_n,X_n^\pi,\mathscr{Y}_n^\pi(n),\mathscr{Z}_n^\pi(k))h+(\mathscr{Z}_k^\pi(n))^\top\Delta W_n,
    \\[6pt]
    \mathscr{Y}_n^\pi(n)=\Y(t_n,X_n^\pi),\quad \mathscr{Z}_k^\pi(n)=\Z(t_k,t_n,X_n^\pi),\quad (n,k)\in\Delta^\pi[0,T]^2.
\end{cases}
\end{equation}
Note that the above is equivalent to~\eqref{eq:BSVIE_disc}, with $\mathscr{Y}_k^\pi(k)=Y_k^\pi$. We emphasize that the scheme presented above is included purely for illustration and is not intended for practical use in this form. 

The following lemma provides a stability estimate between two solutions to~\eqref{eq:BSDE_family_disc_generic}.
\begin{lemma}\label{lemma:main_res}
  Let Assumptions~\ref{ass:1}-\ref{ass:2} hold true and assume that $h$ is small enough so that  $(2K_1 + \tfrac{1}{2})\,h<1$. 
For $j \in \{1,2\}$, suppose 
  $\bigl\{\widehat{Y}_n^{k,\pi,j}, \widehat{Z}_n^{k,\pi,j}\bigr\}_{(k,n)\in \Delta^\pi[0,T]^2}$ 
  are two square integrable solutions to \eqref{eq:BSDE_family_disc_generic}. 
  Define the differences
  \begin{equation*}
    \delta Y_n^k \;=\; \widehat{Y}_n^{k,\pi,1} \;-\; \widehat{Y}_n^{k,\pi,2},
    \qquad
    \delta Z_n^k \;=\; \widehat{Z}_n^{k,\pi,1} \;-\; \widehat{Z}_n^{k,\pi,2}.
  \end{equation*}
  Then there exist constants $C_Y$ and $C_Z$, depending only on $T$ and $K_1$, such that 
  for every $(k,n)\in \Delta^\pi[0,T]^2$, the following estimates hold:
  \begin{align*}
    \E\bigl|\delta Y_n^k\bigr|^2 
    &\;\;\le\;\; 
      C_Y \Bigl(\,
        \E\bigl|\delta Y_N^k\bigr|^2 
        \;+\; 
        \sum_{\ell=n}^{N-1}\E\bigl|\delta Y_N^\ell\bigr|^2\,h
      \Bigr),
    \\[6pt]
    \E\bigl|\delta Z_n^k\bigr|^2 \, h
    &\;\;\le\;\; 
      C_Z \Bigl(\,
        \E\bigl|\delta Y_N^k\bigr|^2 
        \;+\; 
        \sum_{\ell=n}^{N-1}\E\bigl|\delta Y_N^\ell\bigr|^2\,h
      \Bigr).
  \end{align*}
    \begin{proof}
We first derive a discrete Grönwall inequality for the error terms
\(\mathbb E|\delta Y_n^k|^2\) and \(\mathbb E|\delta Z_n^k|^2\).
Inspecting the recursion shows that, once the time grid is fine
enough, i.e., for \(N \ge (2K_1 +\frac{1}{2})T\), where
\(K_1\) is the Lipschitz-growth constant in Assumption \ref{ass:1} and
\(T\) is the time horizon, the associated
constants are monotone decreasing in \(N\) and therefore uniformly
bounded.  Because the \(Z\)-process carries two time indices, we treat
the cases \(n>k\) and \(n=k\) separately before combining the
estimates.  Full algebraic details are given in
Appendix~\ref{app:proof_41}. \qedhere
    \end{proof}
\end{lemma}

We are now ready to prove Theorem~\ref{thm:main_result}.
\begin{proof}[Proof of Theorem~\ref{thm:main_result}]
We decompose the overall error into three main contributions:
\begin{enumerate}
    \item[\textit{(i)}] the approximation error arising from replacing the original BSVIE with a family of BSDEs,
    \item[\textit{(ii)}] the discretization error incurred when approximating the BSDE family with a backward Euler--Maruyama scheme,
    \item[\textit{(iii)}] the error due to the difference between the backward discretization scheme for the BSDE family and our BSVIE discretization scheme.
\end{enumerate}
Denote these contributions by $\text{Err}_1(h)$, $\text{Err}_2(h)$, and $\text{Err}_3(h)$, respectively. 
Then we have
\begin{align*}
        &\int_0^T\E|Y_t-Y_{\pi(t)}^\pi|^2\d t + \int_0^T\int_t^T\E|Z_{t,s}-Z_{\pi(t),\pi(s)}^\pi|^2\d s\d t\leq\text{Err}_1(h)+\text{Err}_2(h)+\text{Err}_3(h),
\end{align*}
where
\begin{align*}
            \text{Err}_1(h)&=
        \int_0^T\E|Y_t-Y_t^{\pi(t)}|^2\d t + \int_0^T\int_t^T\E|Z_{t,s}-Z_s^{\pi(t)}|^2\d s\d t,\\
        \text{Err}_2(h)&=\int_0^T\E|Y_t^{\pi(t)}-Y_{\pi(t)}^{\pi(t),\pi}|^2\d t + \int_0^T\int_t^T\E|Z_s^{\pi(t)}-Z_{\pi(t)}^{\pi(s),\pi}|^2\d s\d t,\\
\text{Err}_3(h)&=\int_0^T\E|Y_{\pi(t)}^{\pi(t),\pi}-Y_{\pi(t)}^\pi|^2\d t + \int_0^T\int_t^T\E|Z_{\pi(t)}^{\pi(s),\pi}-Z_{\pi(t),\pi(s)}^\pi|^2\d s\d t.
\end{align*}
From Theorem~\ref{thm:wang_etc_2}, we have $\text{Err}_1(h)\leq C_1h$. A slight re-write of $\text{Err}_2(h)$ and applying Theorem~\ref{thm:BSDE_disc_conv} yield
\begin{align*}
    \text{Err}_2(h)=\sum_{k=0}^{N-1}\Big(\int_{t_k}^{t_k+1}\E|Y_t^{\pi(t)}-Y_{\pi(t)}^{\pi(t),\pi}|^2\d t + h\int_{t_k}^{T}\E|Z_s^{\pi(t)}-Z_{\pi(t)}^{\pi(s),\pi}|^2\d s\Big)\leq C_2h.
\end{align*}
For $\text{Err}_3(h)$, we note that
\begin{align*}
    \text{Err}_3(h)=\sum_{n=0}^{N-1}\Big(\E|Y_n^{n,\pi}-Y_n^\pi|^2h + h\sum_{k=n}^{N-1}\E|Z_n^{k,\pi}-Z_{n,k}^\pi|^2\d s\Big),
\end{align*}
to which we want to apply Lemma~\ref{lemma:main_res}. Define the two discrete processes as follows. First, set $\big\{Y_n^{n,\pi,1},Z_n^{k,\pi,1}\big\}_{(k,n)\in\Delta^\pi[0,T]^2}\allowbreak=\big\{Y_n^{n,\pi},Z_n^{k,\pi}\big\}_{(k,n)\in\Delta^\pi[0,T]^2}$, where $\big\{Y_n^{n,\pi},Z_n^{k,\pi}\big\}$ is produced by the backward Euler--Maruyama scheme~\eqref{eq:BSDE_family_disc}. Next, set $\bigl\{Y_n^{n,\pi,2},Z_n^{k,\pi,2}\bigr\}_{(k,n)\in\Delta^\pi[0,T]^2}=\bigl\{\mathscr{Y}_{n}^\pi(n),\mathscr{Z}_{n}^\pi(k)\bigr\}_{(k,n)\in\Delta^\pi[0,T]^2}$, where $\big\{\mathscr{Y}_{n}^\pi(n),\mathscr{Z}_{n}^\pi(k)\big\}$ is obtained from the discretization scheme for the BSVIE~\eqref{eq:BSVIE_disc}, using the notation in \eqref{eq:BSVIE_BSDE_disc}. Then
\begin{align*}
    \text{Err}_3(h)\leq&C_Y \sum_{n=0}^{N-1}\Bigl(\,
        \E\bigl|\delta Y_N^n\bigr|^2 
        \;+\; 
        \sum_{\ell=n}^{N-1}\E\bigl|\delta Y_N^\ell\bigr|^2\,h
      \Bigr) +   C_Z \sum_{n=0}^{N-1}h\sum_{k=n}^{N-1}\Bigl(\,
        \E\bigl|\delta Y_N^k\bigr|^2 
        \;+\; 
        \sum_{\ell=n}^{N-1}\E\bigl|\delta Y_N^\ell\bigr|^2\,h
      \Bigr) \\
      \leq & (C_Y+TC_Z)(1+T)\sum_{n=0}^{N-1}\E\bigl|\delta Y_N^n\bigr|^2\\
      =&(C_Y+TC_Z)(1+T)\sum_{n=0}^{N-1}\E\bigl|\mathscr{Y}_N^\pi(n)-g(t_n,X_N^\pi)|^2. 
\end{align*}
Combining the estimates for $\text{Err}_1(h)$, $\text{Err}_2(h)$, and $\text{Err}_3(h)$, we obtain the overall error bound, which completes the proof.
\end{proof}
\color{black}
\begin{remark}
    In addition to the discretization error analyzed in Theorem~\ref{thm:main_result}, practical implementations involve other sources of error, such as optimization error, representation error due to finite network capacity, and sampling error from Monte Carlo approximations. Our analysis does not cover these contributions; instead, we focus on establishing the connection between the time discretization and the residual loss. Empirically, we find that the network can reduce the loss sufficiently to observe the predicted convergence order, but a rigorous justification in the spirit of \cite{han2020convergence} remains an important direction for future work.
\end{remark}
\color{black}
\section{Fully implementable scheme and neural network details}\label{sec:fully_NN}
In this section, we present a fully discretized problem formulation and introduce neural networks as function approximators.

\subsection{Fully implementable algorithms}\label{sec:fully_implementable}
Without further specifications, \eqref{eq:var_disc_FBSVIE} assumes exact optimization over an unspecified set of functions $\Y$ and $\Z$ and the exact computation of expectations. To define a fully implementable scheme, we introduce the parametric functions $\Y^\theta\colon[0,T]\times\R^d\to\R$ and $\Z^\theta\colon\Delta[0,T]^2\times\R^d\times\R^d\to\R^\ell$. Here $\Y^{\theta_y}$ and $\Z^{\theta_z}$ are neural networks and $\theta=(\theta_y,\theta_z)$ represents all trainable parameters. We assume that $\theta$ takes values in some parameter space $\Theta$. Moreover, expectations are approximated with batch Monte Carlo simulation. Let $K_{\mathrm{epoch}}\geq1,K_{\mathrm{batch}}\in\N$ be the number of epochs and the number of batches per epoch, respectively. Let further $M_{\mathrm{train}},M_\text{batch}\in\N$ be the size of the training data set and batch, respectively. We assume that $M_{\mathrm{train}}/M_\text{batch}=K_{\mathrm{batch}}\in\N$. \color{black} The training data consist of $M_{\mathrm{train}}$ independent paths of Wiener increments $\{\Delta W_k(m)\}_{k=0}^{N-1}$, $m=1,\ldots,M_{\mathrm{train}}$, which are reshuffled at the start of each epoch and partitioned into $K_{\mathrm{batch}}$ disjoint mini-batches of size $M_{\text{batch}}$.
\color{black} The training is initialized by random sampling of $\theta^0\in\Theta$. \color{black} Concretely, we iterate over epochs $e=1,\ldots,K_{\mathrm{epoch}}$ and, within each epoch, over mini-batches $b=1,\ldots,K_{\mathrm{batch}}$, updating $\theta$ at each $(e,b)$ by one optimizer step on the mini-batch objective \eqref{eq:var_fully_disc_FBSVIE}.
\color{black} For each update step in an epoch of the training algorithm, we take $M_\text{batch}$ Wiener increments $\Delta W_0(m),\ldots,\Delta W_{N-1}(m)$, $m=1,2,\dots,M_\text{batch}$ from the training data set that were not previously used during the epoch and update $\theta$ by approximate optimization of the following problem:

\begin{equation}\label{eq:var_fully_disc_FBSVIE}\begin{dcases}
\underset{\color{black}\theta\in\Theta\color{black}}{\mathrm{minimize}}\ \sum_{n=0}^{N-1}
 \frac{1}{M_\text{batch}}\sum_{m=1}^{M_\text{batch}}\big|\mathscr{Y}_N^{\pi,\theta}(n)(m)-g\big(t_n,X_n^{\pi,\theta}(m),X_N^{\pi,\theta}(m)\big)\big|^2,\\ 
 \text{For}\  m=1,\ldots,M_\text{batch}:\\
X_n^{\pi,\theta}(m)=x_0+\sum_{k=0}^{n-1}b\big(t_k,X_k^{\pi,\theta}(m),Y_k^{\pi,\theta}(m), Z_{k,k}^{\pi,\theta}(m)\big)h,\\
\ \ \ \qquad\qquad+ \sum_{k=0}^{n-1}\sigma\big(t_k,X_k^{\pi,\theta}(m),Y_k^{\pi,\theta}(m)\big)\Delta W_k(m),\\
\mathscr{Y}_N^{\pi,\theta}(n)(m)=Y_n^{\pi,\theta}(m) - \sum_{k=n}^{N-1}f(t_n,t_k,X_k^{\pi,\theta}(m),Y_k^{\pi,\theta}(m),Z_{n,k}^{\pi,\theta}(m),Z_{k,k}^{\pi,\theta}(m))h \\
\qquad\qquad\qquad\ + \sum_{k=n}^{N-1}(Z_{n,k}^{\pi,\theta}(m))^\top\Delta W_k(m),\\
Y_n^{\pi,\theta}(m)=\Y^\theta(t_n,X_n^{\pi,\theta}(m)), \quad Z_{n,k}^{\pi,\theta}(m)=\Z^\theta(t_n,t_k,X_n^{\pi,\theta}(m),X_k^{\pi,\theta}(m)),\\
    \end{dcases}
\end{equation}

When all training data has been used, a new epoch starts. After $K_{\mathrm{epochs}}$ epochs, the algorithm terminates.
The neural network parameters at termination are $\theta^*$. It is an approximation of the parameters $\theta^{**}$ that optimize \eqref{eq:var_fully_disc_FBSVIE} in the limit $M_\text{batch}\to\infty$. \color{black}Hence, the discrepancy between $(Y^{\pi,\theta^{**}},Z^{\pi,\theta^{**}})$ and the solution to
\eqref{eq:var_disc_FBSVIE} is governed by the representation error. If the chosen neural network
classes are dense in the target function spaces, this error can vanish as model capacity grows.
\color{black}

\subsection{Specification of the neural networks}\label{sec:spec_NN}
Here, we introduce the neural networks that we use in our implementations in Section~\ref{sec:fully_implementable}. The generality is kept to a minimum and more general architectures are of course possible. For $\Y^\theta\colon[0,T]\times\R^d\to\R$ and $\Z^\theta\colon\Delta[0,T]^2\times\R^d\times\R^d\to\R^\ell$, we employ fully-connected, feed-forward networks with three hidden layers; because the input dimension of $\Z^\theta$ is larger than that of $\Y^\theta$, we use 100 neurons in each hidden layer for $\Z^\theta$ and 50 neurons in each hidden layer for $\Y^\theta$. In both architectures, each affine transformation in the hidden layers is followed by the element-wise ReLU activation function, $\mathfrak{R}(x)=\max(0,x)$, while the output layer remains unactivated. More precisely, for $x,x_t\in\R^d$ and $(t,s)\in\Delta[0,T]^2$ denote by $\mathbf{x}_y=\text{Concat}(t,x)\in\R^{d+1}$ and $\mathbf{x}_z=\text{Concat}(t,s,x_t,x)\in\R^{2d+2}$
\begin{align*}
    \Y^{\theta_y}(\mathbf{x}_y)
    &=
    \mathcal{W}_y^4\mathfrak{R}(\mathcal{W}_y^3\mathfrak{R}(\mathcal{W}_y^2\mathfrak{R}(\mathcal{W}_y^1\mathbf{x}_y+b_y^1)+b_y^2)+b_y^3)+b_y^4,\\
        \Z^{\theta_y}(\mathbf{x}_z)
    &=
    \mathcal{W}_z^4\mathfrak{R}(\mathcal{W}_z^3\mathfrak{R}(\mathcal{W}_z^2\mathfrak{R}(\mathcal{W}_z^1\mathbf{x}_z+b_z^1)+b_z^2)+b_z^3)+b_z^4,
\end{align*}
with weight matrices 
$\mathcal{W}_y^1\in\R^{50\times (d+1)}$,  $\mathcal{W}_y^2,\mathcal{W}_y^3\in\R^{50\times 50}$, $\mathcal{W}_y^4\in\R^{1\times 50}$, and $\mathcal{W}_z^1\in\R^{100\times (2d+2)}$,  $\mathcal{W}_z^2,\mathcal{W}_z^3\in\R^{100\times 100}$, $\mathcal{W}_z^4\in\R^{\ell\times 100}$
and bias vectors
$b_y^1,b_y^2,b_y^3\in\R^{50}$,
 $b_y^4\in\R$, and $b_z^1,b_z^2,b_z^3\in\R^{100}$,
 $b_z^4\in\R^\ell$.
Finally, denote by $\theta_y=(\mathcal{W}_y^1,\mathcal{W}_y^2,\mathcal{W}_y^3,\mathcal{W}_y^4,b_y^1,b_y^2,b_y^3,b_y^4)$, $\theta_y=(\mathcal{W}_z^1,\mathcal{W}_z^2,\mathcal{W}_z^3,\mathcal{W}_z^4,b_z^1,b_z^2,b_z^3,b_z^4)$, and $\theta=(\theta_y,\theta_z)$ where the matrices are considered vectorized before concatenation.

\section{Numerical experiments}\label{sec:numerical}
This section is divided into two parts.  Section~\ref{sec:num_1} focuses on test problems that satisfy the assumptions of Section~\ref{sec:error_analysis}, whereas Section~\ref{sec:num_2} relaxes those assumptions and investigates coupled FSDE-BSVIEs.

Throughout all experiments we adopt the same hyper-parameter configuration.  The mini-batch size is fixed at $M_{\text{batch}} = 2^{11}$, and the total number of training paths at $M_{\text{train}} = 2^{18}$.  Each path is therefore processed ten times, giving $K_\text{epoch}=10$ training epochs with random shuffling between epochs. The learning rate is initialized at $0.005$ and follows an exponential decay schedule, being multiplied by $\mathrm{e}^{-0.2}$ after every epoch.  Optimization is carried out with the Adam algorithm~\cite{kingma2014adam}. Further implementation details can be found at \url{https://github.com/AlessandroGnoatto/DeepBSVIE}.

Although we report results with a single feed-forward architecture for all problems, during development we experimented with a wide range of hyper-parameters: between $1$ and~$6$ hidden layers, $10$--$300$ neurons per layer, ReLU versus $\tanh$ activations, and both Adam and SGD optimizers. Once the network had sufficient capacity, the solver’s accuracy and empirical convergence rate were essentially insensitive to the exact depth, width, or activation choice.  We therefore settled on a network with three hidden layers and $100$ neurons per layer, oversized for the low-dimensional examples yet still tractable in up to 500 dimensions, which performed on par with both slimmer and deeper alternatives. This configuration is thus reported as the smallest architecture that remained robust across all test cases.

\subsection{Examples where the error analysis apply}\label{sec:num_1}
In this section, we consider two decoupled FSDE–BSVIE systems in which both the driver $f$ and the free term $g$ depend explicitly on the time variable $t$. That is, we study systems of the form
\begin{equation}
\begin{cases}
X_t \;=\; x_0 \;+\; \displaystyle \int_0^t  b(s,X_s)\d s + \int_0^t\sigma(s, X_s)\d W_s ,
\\[6pt]
Y_t 
\;=\;
g(t,X_T) +\displaystyle \int_t^Tf(t,s,X_s,Y_s,Z_{t,s})\d s -\int_t^TZ_{t,s}^\top \d W_s,\quad
(t,s)\in\Delta[0,T]^2.
\end{cases}
\end{equation}
Assuming that the coefficients satisfy the required regularity conditions, our numerical analysis is applicable to this class of systems.

We consider two examples: one in which the FSDE is driven by additive noise, and another in which it is driven by multiplicative noise.
\subsubsection{Example 1A: Additive noise}
Let $d\in\N$, $k\in\R$, $\mu,x_0\in\R^d$ and $\sigma\in\R^{d\times d}$ be constants with $\sigma$ invertible. We consider the following system
\begin{equation}
\begin{cases}
X_t \;=\; x_0 \;+\; \displaystyle  \mu\,t \;+\;  \sigma\,W_t,
\quad \overline{X}_t \;=\; \frac{1}{d}\sum_{i=1}^d X_t^{i},
\\[6pt]
Y_t 
\;=\;
t\,\sin\!\bigl(k\overline{X}_T\bigr)
\;+\;
\displaystyle \int_{t}^T
\Bigl(\,
\tfrac{tk^2}{2d^2}\,\sin\!\bigl(k\overline{X}_s\bigr)\,\|\sigma\|^2
\;-\;
\mu^\top \sigma^{-1} \,Z_{t,s}
\Bigr)\,\d s
\;-\;
\displaystyle \int_{t}^T
Z_{t,s}^\top\,\d W_s,\quad
(t,s)\in\Delta[0,T]^2.
\end{cases}
\end{equation}
A direct calculation shows that the unique solution $\{(Y_t,Z_{t,s})\}_{t\le s\le T}$ is given by the closed-form expressions
\begin{equation}\label{eq:ref_solution_BSVIE_1A}
Y_t \;=\; t \sin(k\overline{X}_t)
\quad\text{and}\quad
Z_{t,s} 
\;=\;\frac{tk}{d}\cos(k\overline{X}_s)\sigma \mathbf{1}_d,
\quad
(t,s)\in\Delta[0,T]^2.
\end{equation}
Here, $\mathbf{1}_d$ is a d-dimensional vector consisting of ones.

Let $k=5$, $d=5$, $T=1$, $x_0=(1,1,1,1,1)^\top$, $\mu=\text{diag}(0.25,0.25,0.25,0.25,0.25)$, and $\sigma=\text{diag}(0.8,\allowbreak 0.9,\allowbreak 1.0,\allowbreak 1.1, \allowbreak 1.2 \allowbreak)$.

Figure~\ref{fig:Y_5d} displays the approximate $Y$-process alongside the analytical reference solution. In the left frame, three representative sample paths are shown, while the right frame presents the sample mean together with the 25th and 75th sample percentiles.

\begin{figure}[htp]
\centering
\begin{tabular}{cc}
\includegraphics[width=80mm]{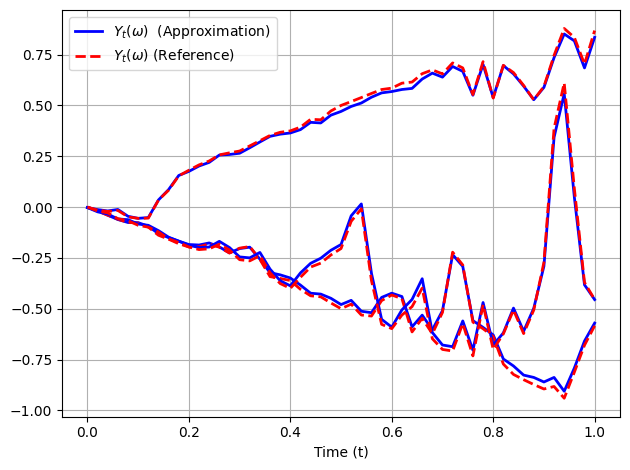}& \includegraphics[width=80mm]{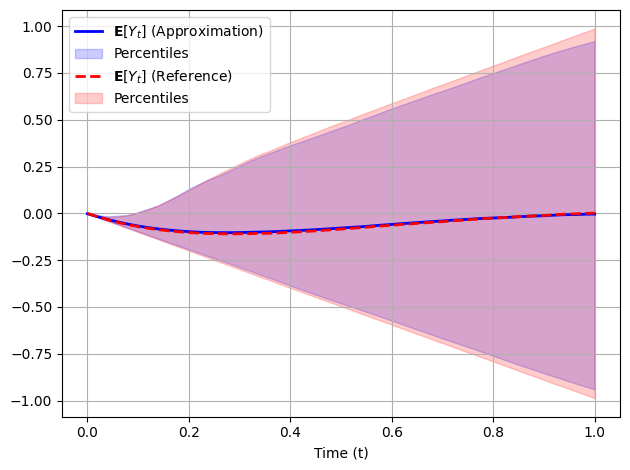}
\end{tabular}
\caption{Comparison of the approximated $Y$ with the reference solutions for Example~1A. \textbf{Left:} Three representative sample paths. \textbf{Right:} The sample mean and the 25th and 75th percentiles.
}\label{fig:Y_5d}
\end{figure}

Figure~\ref{fig:Z_5d} displays the first component of the $Z_{t,\cdot}$-process for different values of $t$ compared with the analytical reference solutions for Example~1A. In the left frame, one representative sample path is shown, while the right frame presents the sample mean.

\begin{figure}[htp]
\centering
\begin{tabular}{c}
\includegraphics[width=160mm]{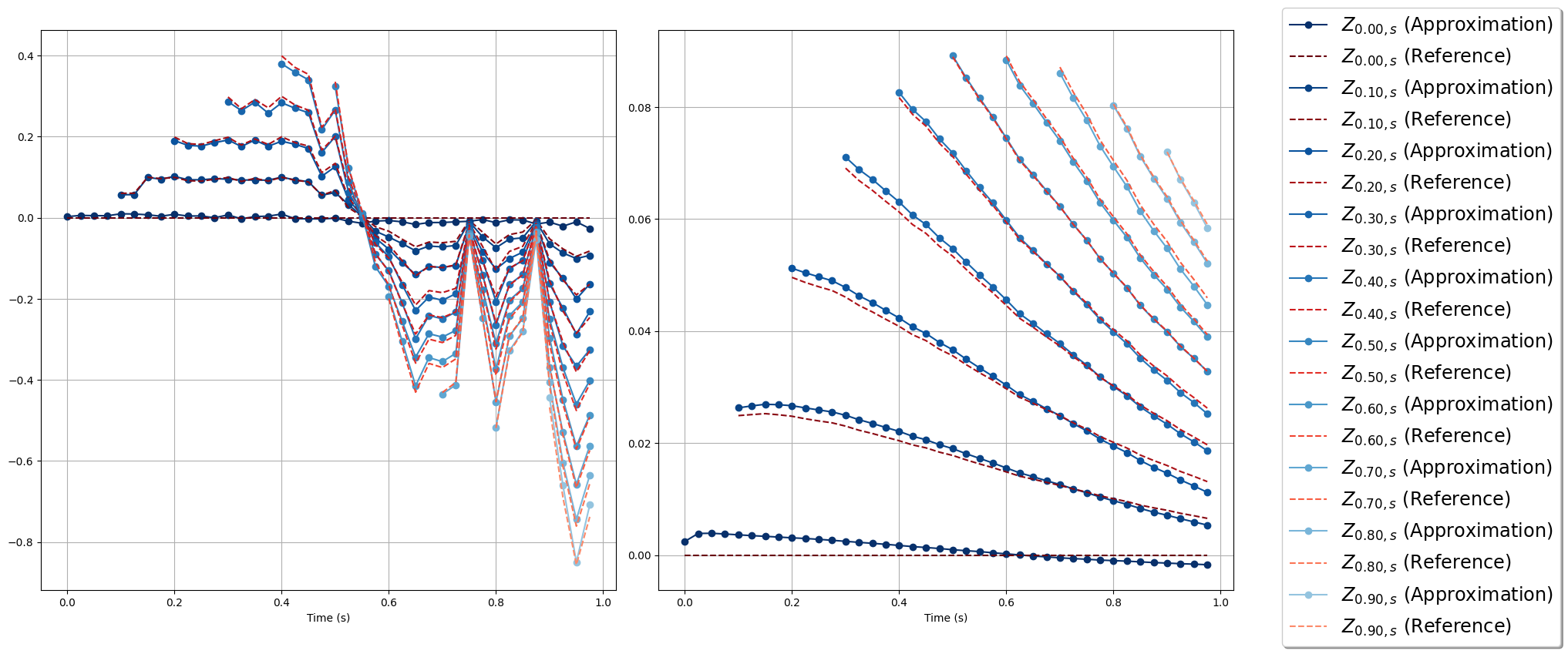}
\end{tabular}
\caption{Comparison of the approximated $Z_{t,s}$ with the reference solution for different values of $t$ for Example~1A. \textbf{Left: } One representative sample path of the first (of 5) component of $Z_{t,s}$. \textbf{Right: } A sample mean for the first component of $Z_{t,s}$.}\label{fig:Z_5d}
\end{figure}

\subsubsection{Example 1B: Multiplicative noise}
Let $d\in\N$, $k\in\R$, $\mu,x_0\in\R^d$ and $\sigma\in\R^{d\times d}$ be constants with $\sigma$ invertible. We consider the following system
\begin{equation}\label{eq:BSVIE_1b}
\begin{cases}
X_t \;=\; x_0+\displaystyle  \int_0^t\text{diag}(\mu) X_s\,\d s \;+\; \int_0^t\text{diag}(\sigma\,X_s) \d W_s,\quad \overline{X}_t=\frac{1}{d}\sum_{i=1}^dX_t^{i},
\\[6pt]
Y_t 
\;=\;
t\,\sin\!\bigl(k\overline{X}_T\bigr)
\;+\;
\displaystyle \int_{t}^T
\bigl(\tfrac{tk^2}{2d^2}\,\sin\!\bigl(k\overline{X}_s\bigr)\,\|\sigma\,X_s\|^2
\;-\;
\mu^\top\sigma^{-1}Z_{t,s}
\bigr)\d s
\;-\;
\displaystyle \int_{t}^T
Z_{t,s}^\top
\,\d W_s,\ 
(t,s)\in\Delta[0,T]^2.
\end{cases}
\end{equation}
A direct calculation shows that the unique solution $\{(Y_t,Z_{t,s})\}_{(t,s)\in\Delta[0,T]^2}$ is given by the closed-form expressions
\begin{equation}\label{eq:ref_solution_BSVIE_1B}
Y_t \;=\; t \sin(k\overline{X}_t)
\quad\text{and}\quad
Z_{t,s} 
\;=\;\frac{tk}{d}\cos(k\overline{X}_s)\sigma X_s,
\quad
(t,s)\in\Delta[0,T]^2.
\end{equation}
Let $d=5$, $T=1$, $x_0=(1,1,1,1,1)^\top$, $\mu=\text{diag}(0.05,0.05,0.05,0.05,0.05)$, and $\sigma=\text{diag}(0.2,\allowbreak 0.25,\allowbreak 0.3,\allowbreak 0.35, \allowbreak 0.45 \allowbreak)$.
Figure~\ref{fig:Y_5d_GBM} displays the approximate $Y$-process alongside the analytical reference solution. In the left frame, three representative sample paths are shown, while the right frame presents the sample mean together with the 5th and 95th sample percentiles. \color{black}Since the variance in Example 1B is smaller than in Example 1A, we use the 5th and 95th percentiles in Example 1B to better illustrate the spread of the distribution, whereas in Example 1A the 25th and 75th percentiles already provide a clear view of the central tendency.\color{black}
\begin{figure}[htp]
\centering
\begin{tabular}{cc}
\includegraphics[width=80mm]{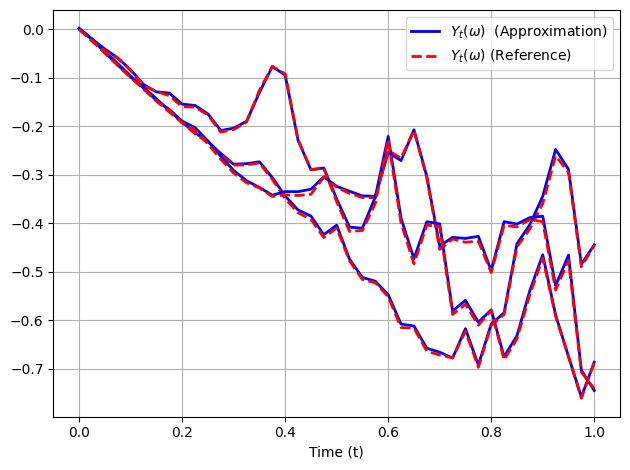}& \includegraphics[width=80mm]{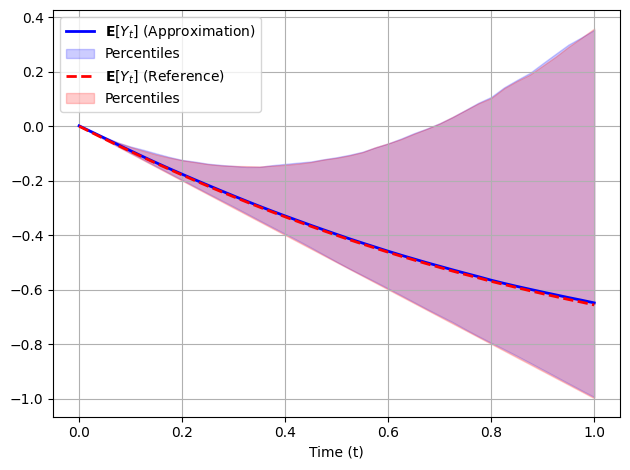}
\end{tabular}
\caption{Comparison of the approximated $Y$ with the reference solutions for Example~1B. \textbf{Left:} Three representative sample paths. \textbf{Right:} The sample mean and the 5th and 95th percentiles.
}\label{fig:Y_5d_GBM}
\end{figure}

Figure~\ref{fig:Z_5d_GBM} displays the first component of the $Z_{t,\cdot}$-process for different values of $t$ compared with the analytical reference solutions. In the left frame, one representative sample path is shown, while the right frame presents the sample mean.

\begin{figure}[htp]
\centering
\begin{tabular}{c}
\includegraphics[width=160mm]{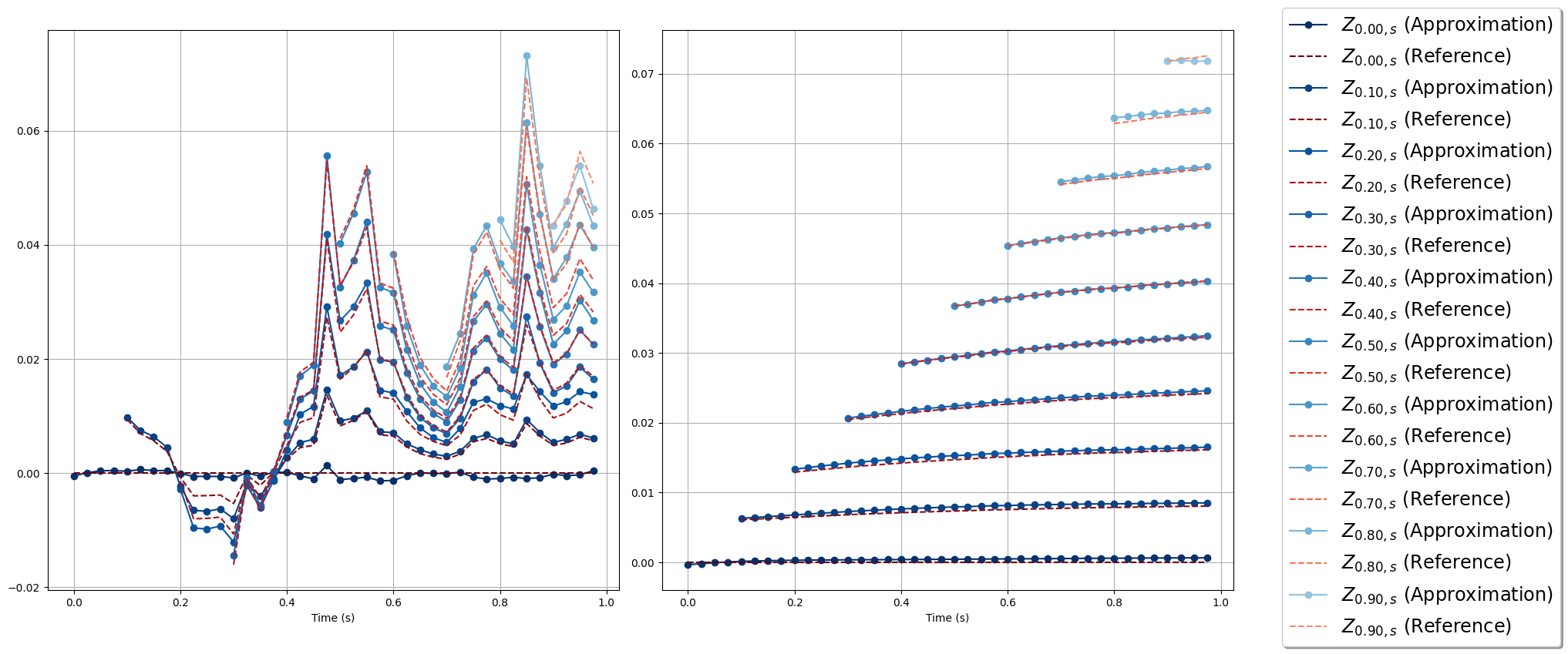}
\end{tabular}
\caption{Comparison of the approximated $Z_{t,s}$ with the reference solution for different values of $t$ for Example~1B. \textbf{Left: } One representative sample path of the first (of 5) component of $Z_{t,s}$. \textbf{Right: } A sample mean the first component of $Z_{t,s}$.}\label{fig:Z_5d_GBM}
\end{figure}

\subsubsection{Example 1: Empirical error analysis}
\color{black}We define the simulation errors and the post-optimization value of the loss function, respectively, by
\begin{align}
\label{eq:sim_errors}
\begin{split}
    \mathrm{Err}^Y(N,M)&\defeq\frac{1}{M}\sum_{m=1}^M\sum_{n=0}^{N-1}\Big|Y_{t_n}(m)-Y_n^{\pi,\theta^*}(m)\Big|^2h,\\
     \mathrm{Err}^Z(N,M)&\defeq\frac{1}{M}\sum_{m=1}^M\sum_{k=0}^{N-1}\sum_{n=k}^{N-1}\Big|Z_{t_k,t_n}(m)-Z_{k,n}^{\pi,\theta^*}(m)\Big|^2h^2,
     \\
     \mathrm{Err}^{\text{T}}(N,M)&\defeq\frac{1}{M}\sum_{m=1}^{M}\sum_{n=0}^{N-1}
 \Big|\mathscr{Y}_N^{\pi,\theta^*}(n)(m)-g\big(t_n,X_n^{\pi,\theta^*}(m),X_N^{\pi,\theta^*}(m)\big)\Big|^2h.
 \end{split}
\end{align}
\color{black} 
Figure~\ref{fig:conv_plot} displays the empirical convergence of our approximation in terms of the stepsize $h$ (or the number of discretization steps $N$). We \color{black}fix \(M=2^{12}\), and \color{black} choose $N \in \{10,20,30,40,50\}$ for Example 1A and $N \in \{10,20,30,40\}$ for Example 1B, where the variance of the solution is lower. 
\begin{figure}[htp]
\centering
\begin{tabular}{cc}
\includegraphics[width=80mm]{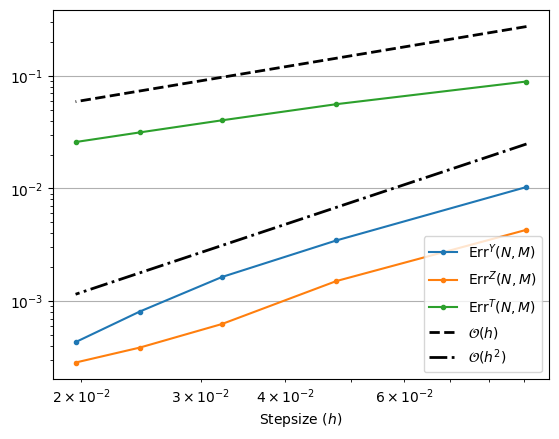}&
\includegraphics[width=80mm]{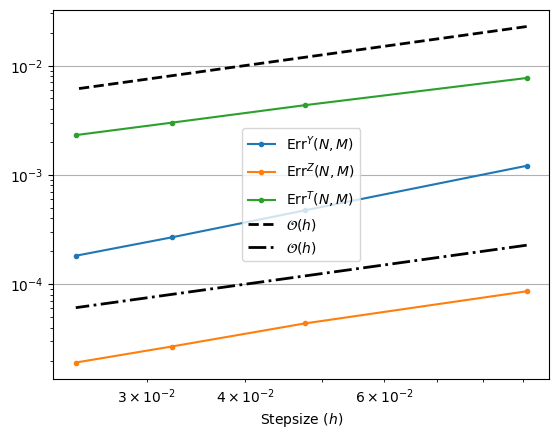}
\end{tabular}
\caption{Empirical convergence plot for our approximate $Y$, $Z$ and the \color{black}post-optimization value of the loss function\color{black}. \textbf{Left:} Example 1A. \textbf{Right:} Example 1B.}\label{fig:conv_plot}
\end{figure}
Note that in Example 1A, where the FSDE is an arithmetic Brownian motion, the Euler--Maruyama scheme coincides with the exact solution. This implies that the entire discretization error is attributable to the discretization of the BSVIE. For Example 1B, the FSDE is a geometric Brownian motion, for which the Euler--Maruyama scheme has a strong discretization error of order $0.5$. Moreover, we have access to a closed-form solution for the geometric Brownian motion, which is used when the reference solution is computed (while the Euler--Maruyama scheme is employed for the FSDE in our approximate solution).

In both Example~1A and Example~1B, we observe an empirical convergence order of 1 for the optimization loss. However, for $Y$ and $Z$, the convergence order is 1 in Example~1A and 0.5 in Example~1B. This suggests that employing a higher-order approximation for the FSDE can improve the observed overall error rate.

\subsubsection{Example 1: CPU time and scalability in the spatial dimension}
In this subsection, we study how the method scales with the spatial dimension $d$ in terms of both accuracy and wall-clock runtime. As a test case, we consider the BSVIE~\eqref{eq:BSVIE_1b}. For several values of $d$ we report the wall-clock runtime and the simulation errors errors defined in~\eqref{eq:sim_errors}.
\begin{table}[htbp]
  \centering
  \caption{Numerical results for various dimensions \(d\).}
  \label{tab:dimensions}
  \sisetup{
    detect-all,
    table-number-alignment = center,
    table-figures-integer  = 2,
    table-figures-decimal  = 4,
    table-figures-exponent = 1
  }
\begin{tabular}{c S S c S}
  \toprule
  {d} & {Err(\(Y\))} & {Err(\(Z\))} & {Runtime (s)} \\
  \midrule
     1   & 7.7e-05  & 8.2e-05  & 430 \\
    5   & 9.6e-05   &  2.9e-05  &  480 \\
    20   &  1.8e-4       &    7.9e-06     &    480  \\
    100   &  1.3e-3        &      4.5e-06    &   480   \\
   500   &     8.0e-3     &    2.9e-06      &    480   \\
    \bottomrule
  \end{tabular}
\end{table}

\color{black}
The algorithm runs on a Google Colab instance with an NVIDIA A100 GPU. Because much of the runtime comes from general processing overhead rather than the actual calculations, working with a 500-dimensional state takes nearly the same time as working with just one dimension. We expect that rewriting the loop in a more GPU-efficient style would make the algorithm substantially faster.\color{black}

\subsection{Examples of general FSDE-BSVIE systems}\label{sec:num_2}
In this section, we treat more general forms of FSDE-BSVIEs, where the numerical analysis, in the form given in this paper, no longer applies.
\subsubsection{Example 2: An FSDE-BSVIE system with a quadratic solution}\label{sec:quadratic_BSVIE}
Let $\mu,x_0\in\R^d$ and $\sigma\in\R^{d\times d}$ be constant, with $\sigma$ invertible. We consider the following system
\begin{equation}\label{eq:BSVIE_example}
\begin{dcases}
X_t \;=\; x_0+\int_0^t \mu \, X_s \,\d s \;+\; \int_0^t\text{diag}(\sigma\,X_s)\d W_s,
\\[4pt]
Y_t =
\bigl\langle\,t + X_t,X_T\bigr\rangle
-
\int_{t}^{T}
\mu^\top
\text{diag}\!\bigl(t + X_t\bigr)\,\sigma^{-1}\,
\bigl(\text{diag}\!\bigl(t + X_t\bigr)\bigr)^{-1}\,
Z_{t,s}\,\d s
-
\int_{t}^{T} Z_{t,s}^\top\,\d W_s,
\  (t,s)\in\Delta[0,T]^2.
\end{dcases}
\end{equation}
A direct calculation shows that the unique solution $\{(Y_t,Z_{t,s})\}_{(t,s)\in\Delta[0,T]^2}$ is given by the closed-form expressions
\begin{equation}\label{eq:ref_solution_BSVIE_2}
Y_t \;=\; \bigl\langle\,t + X_t,\;X_t\bigr\rangle
\quad\text{and}\quad
Z_{t,s} 
\;=\;
\text{diag}\bigl(t + X_t\bigr)\,\sigma\,X_s,
\quad
(t,s)\in\Delta[0,T]^2.
\end{equation}

Let $d=20$, $T=1$, $x_0=(1,\ldots,1)^\top$, $\mu=(-0.05,\ldots,-0.05)^\top$, $\sigma=\text{diag}(0.3,\allowbreak 0.375,\allowbreak 0.45,\allowbreak 0.375,\allowbreak 0.3,\allowbreak 0.375,\allowbreak 0.45,\allowbreak 0.375,\allowbreak 0.3,\allowbreak 0.375,\allowbreak 0.45,\allowbreak 0.375,\allowbreak 0.3,\allowbreak 0.375,\allowbreak 0.45,\allowbreak 0.375)$, $N=40$.

Figures~\ref{fig:Y_20d} and \ref{fig:Z_20d_quad} display the approximate $Y$- and $Z$-processes alongside their respective analytical reference solutions. For the $Y$-process, we present three representative sample paths, a sample mean, as well as the 5th and 95th sample percentiles. For the $Z$-process, one representative sample path is shown, together with a sample mean of its first component.

\begin{figure}[htp]
\centering
\begin{tabular}{cc}
\includegraphics[width=80mm]{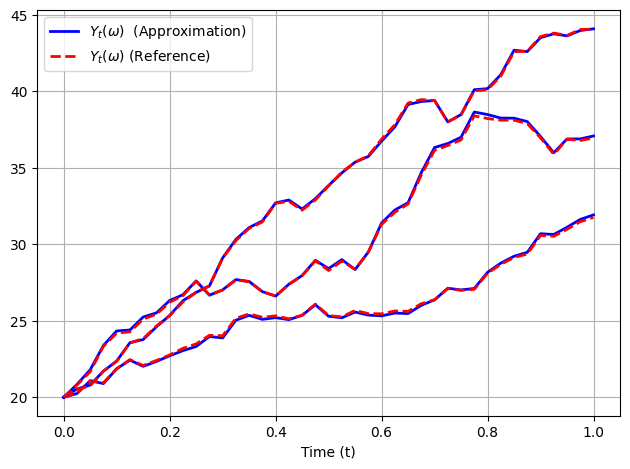}& \includegraphics[width=80mm]{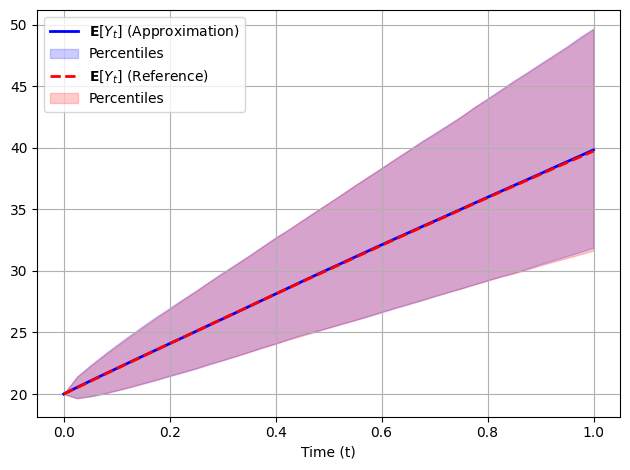}
\end{tabular}
\caption{Comparison of the approximated $Y$ with the reference solutions for Example~2. \textbf{Left:} Three representative sample paths. \textbf{Right:} The sample mean and the 5th and 95th percentiles.
}\label{fig:Y_20d}
\end{figure}

\begin{figure}[htp]
\centering
\begin{tabular}{c}
\includegraphics[width=160mm]{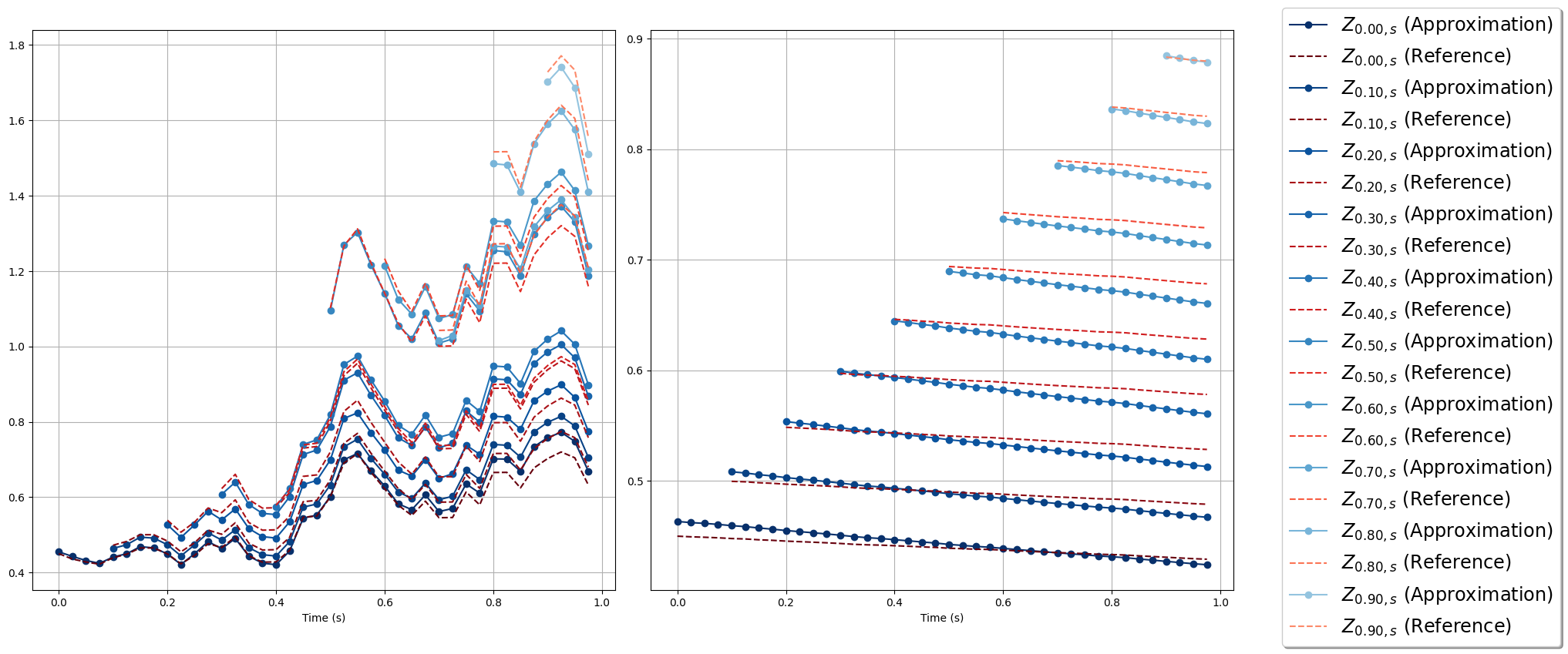}
\end{tabular}
\caption{Comparison of the approximated $Z_{t,s}$ with the reference solution for different values of $t$ for Example~2. \textbf{Left: } One representative sample path of the first (of 20) component of $Z_{t,s}$. \textbf{Right: } A sample mean the first component of $Z_{t,s}$.}\label{fig:Z_20d_quad}
\end{figure}

Because we do not carry out an empirical error analysis in this section, these figures do not offer insight into the accuracy of our approximation for the remaining 19 components of the $Z$-process. Consequently, Figure~\ref{fig:Z_20d_all_components} illustrates a representative sample path for our approximations of $Z_{0,s}$ and $Z_{0.5,s}$, compared to their corresponding analytical reference solutions.

\begin{figure}[htp]
\centering
\begin{tabular}{cc}
\includegraphics[width=80mm]{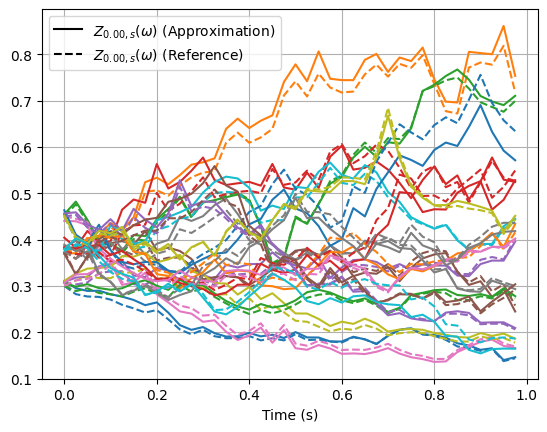}& \includegraphics[width=80mm]{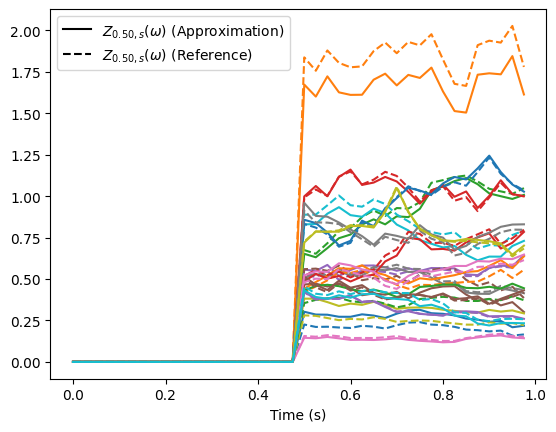}
\end{tabular}
\caption{Comparison of one representative sample path of the approximated $Z_{t,s}$ with its reference solution for Example~2. \textbf{Left:} $t=0$. \textbf{Right:} $t=0.5$. Since $Z_{t,s}$ is only defined for $s \ge t$, we set $Z_{t,s} = \mathbf{0}$ for $s < t$ purely for illustrative purposes.
}\label{fig:Z_20d_all_components}
\end{figure}
\subsubsection{Example 3: An FSDE-BSVIE system coupled in \textit{Y} and \textit{Z}}
For $d\in\N$, $a,x_0\in\R^d$, $b,c,k\in\R$, and $\sigma\in\R^{d\times d}$. We consider the following coupled FSDE-BSVIE system
\begin{equation}\label{eq:BSVIE_fully_coupled}
\begin{dcases}
X_t \;=\; x_0+\int_0^t \big(a+bZ_{s,s}\big)\,\d s \;+\;\int_0^t\big(c+Y_s\big) \sigma W_t,\quad\overline{X}_t=\frac{1}{d}\sum_{i=0}^dX_t^i,
\\[4pt]
Y_t \;=\;
t\sin(k\widebar{X}_T) \;+\;\displaystyle 
\int_{t}^{T}\Big(\frac{tk^2}{2d^2}\sin(k\overline{X}_s)(c+Y_s)^2\|\mathbf{1}^\top\sigma\|^2 - \cos(k\overline{X}_s)(t\mathbf{1}^\top a + sb\mathbf{1}^\top Z_{t,s})\big)\Big)\d s
\\ \;-\;
\int_{t}^{T} Z_{t,s}^\top\,\d W_s,
\quad
(t,s)\in\Delta[0,T]^2.
\end{dcases}
\end{equation}

A direct calculation shows that the unique solution $\{(Y_t,Z_{t,s})\}_{(t,s)\in\Delta[0,T]^2}$ is given by the closed-form expressions
\begin{equation}\label{eq:ref_solution_BSVIE_4}
Y_t \;=\; t\sin(k\widebar{X}_t)
\quad\text{and}\quad
Z_{t,s} 
\;=\;
t\cos(k\overline{X}_s)(c+s\sin(k\overline{X}_s)))\mathbf{1}^\top \sigma,
\quad
(t,s)\in\Delta[0,T]^2.
\end{equation}
Let $d=k=5$, $T=1$, $x_0=1$, $a=(0.15,0.075,0.0,-0.075,-0.15)^\top$, $\sigma=\text{diag}(0.4,0.5,0.6,0.7,0.9)$, $c=1.001$ ($c>1$ to guarantee enough ellipticity), and $N=40$. 

Figure~\ref{fig:XY_fully_coupled}-\ref{fig:Z_fully_coupled} shows our approximate $X$-,$Y$- and $Z$-processes, compared with the semi-analytic (we have to approximate the FSDE with an Euler--Maruyama scheme) reference solutions.  In particular, the reference solution for the FSDE is obtained by substituting \(Y_t = t \sin(k\widebar{X}_t)\), and $Z_{t,t}=\frac{tk}{d}\cos(k\overline{X}_t)(c+t\sin(k\overline{X}_t)))\mathbf{1}^\top \sigma$ into the drift coefficient and approximating the decoupled FSDE via the Euler--Maruyama scheme. 
\begin{figure}[htp]
\centering
\begin{tabular}{cc}
\includegraphics[width=80mm]{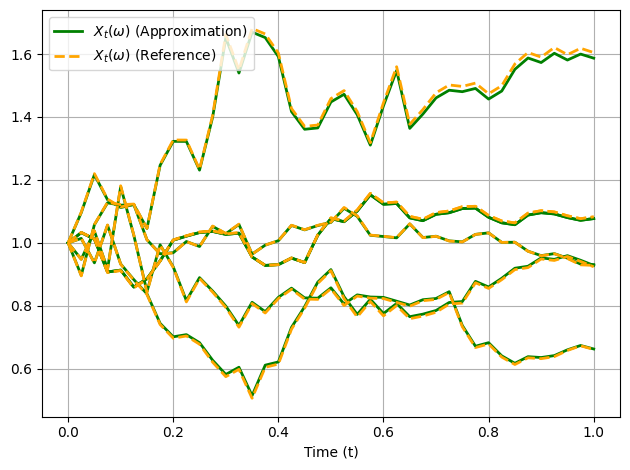}& \includegraphics[width=80mm]{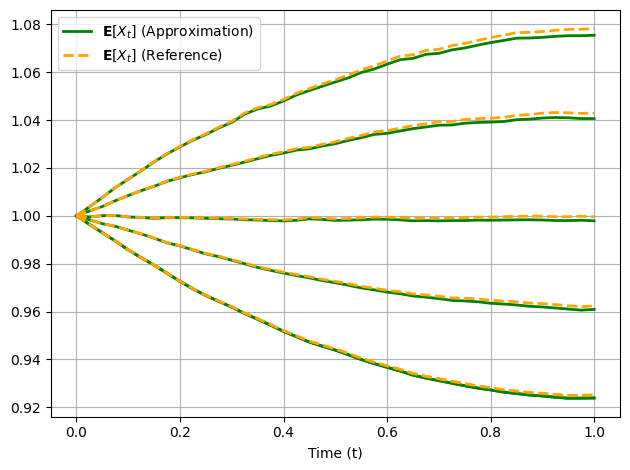}\\
\includegraphics[width=80mm]{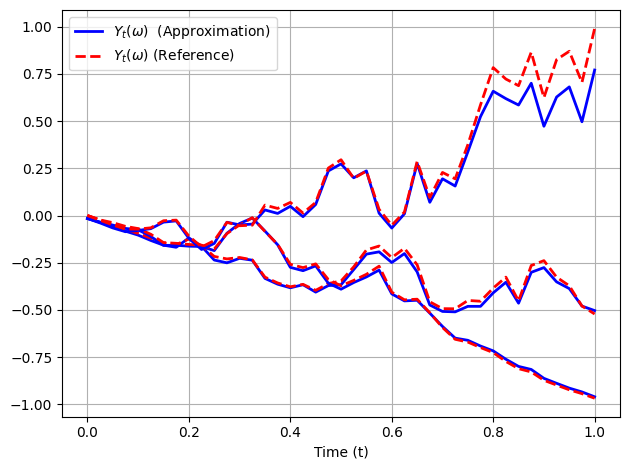}& \includegraphics[width=80mm]{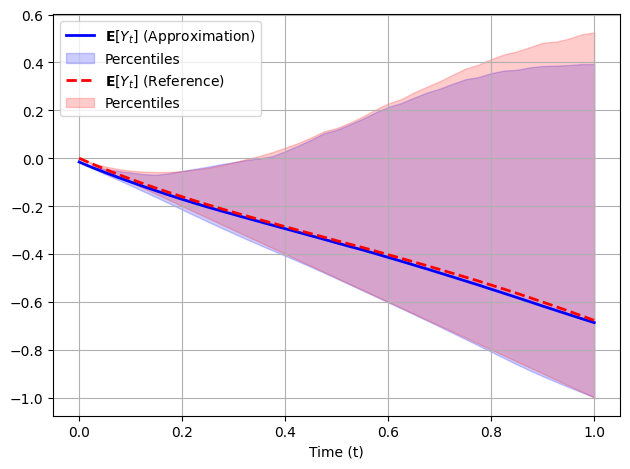}
\end{tabular}
\caption{Comparison of the approximated $X$ and $Y$ with the reference solutions for Example~3. \textbf{Left:} Three representative sample paths. \textbf{Right:} The sample mean and, for $Y$, the 5th and 95th percentiles.
}\label{fig:XY_fully_coupled}
\end{figure}

\begin{figure}[htp]
\centering
\begin{tabular}{c}
\includegraphics[width=160mm]{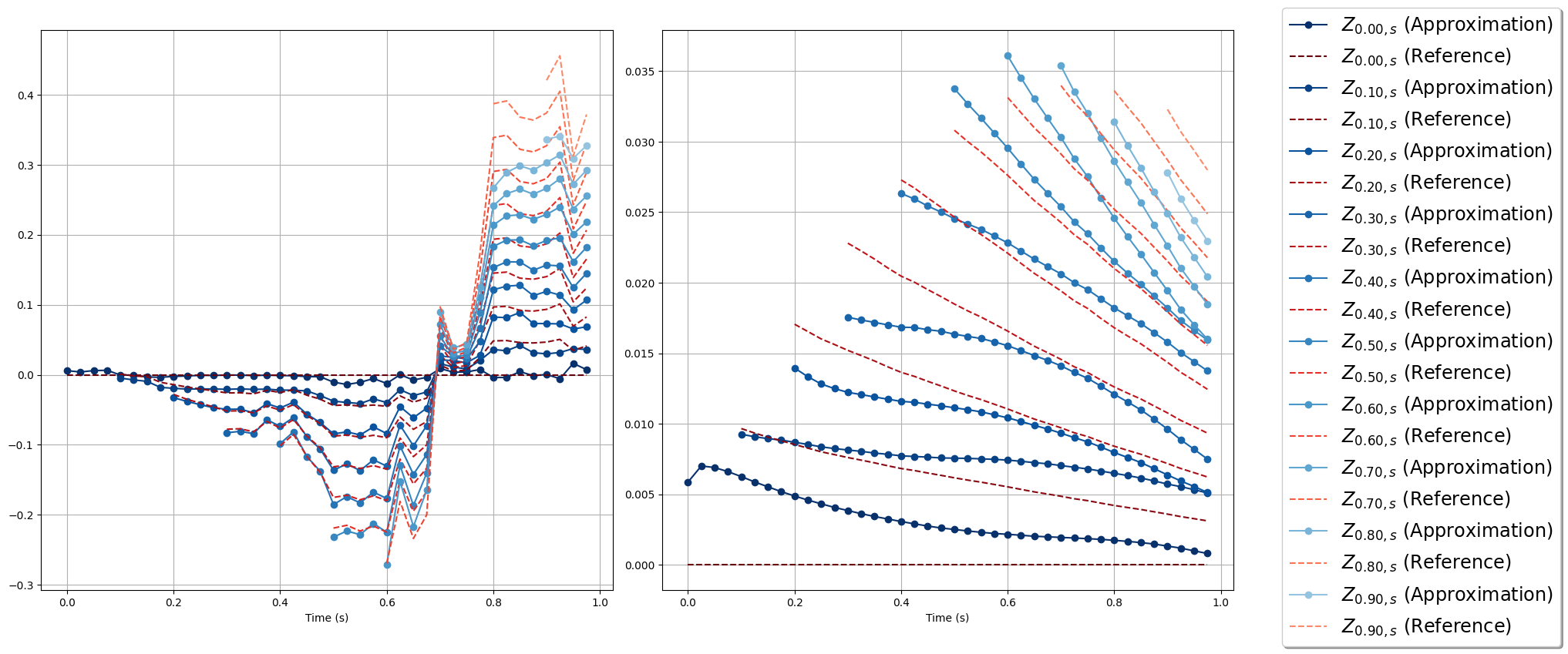}
\end{tabular}
\caption{Comparison of component 1 (of 5) of the approximated $Z_{t,s}$ with the reference solution for different values of $t$ for Example~3. \textbf{Left: } One representative sample path of $Z_{t,s}$. \textbf{Right: } A sample mean of $Z_{t,s}$.}\label{fig:Z_fully_coupled}
\end{figure}

\color{black}
\begin{remark}
It is well known that strongly coupled FBSDEs are difficult to approximate with forward-type deep-learning methods, see for instance \cite{andersson2023convergence,andersson2025deep}. There are, however, recent works establishing convergence for coupled systems under restrictive assumptions, for example, fully coupled McKean--Vlasov FBSDEs in \cite{reisinger2024posteriori} and weakly coupled FBSDEs in \cite{han2020convergence,negyesi2024generalized}. These results typically require the coupling to be of a rather mild nature. With this in mind, caution is needed when applying forward-type methods to strongly coupled FBSDEs or FSDE-BSVIEs. At the same time, our approach differs in two important respects. First, both the $Y$- and $Z$-processes are parametrized directly by neural networks. Second, the loss criterion aggregates residuals at evaluation times $t_k$, each constructed by rolling out the dynamics from $t_k$ to $T$ with $Z_{t_k,\cdot}$. In this way, the loss is local in $t_k$ but global through its multi-step propagation. By contrast, forward deep BSDE solvers perform a single Euler--Maruyama pass from $0$ to $T$ and tune $y_0$ and $Z$ solely to satisfy the terminal condition. Their loss is therefore concentrated at $T$ rather than built from per-$t_k$ rollouts. Whether the convergence issues observed in the FBSDE setting carry over to the BSVIE framework remains an open question.
\end{remark}

\color{black}

\subsection{Comparison with existing methods}
We do not benchmark our method against existing BSVIE solvers because, strictly speaking, no complete solver is yet available: the methods of Wang~\cite{wang2016numerical}, Hamaguchi \& Taguchi~\cite{hamaguchi2023approximations} and Pokalyuk~\cite{pokalyuk2012discretization} provide only the time-grid recursion while leaving the crucial conditional-expectation step unspecified. Implementing those schemes therefore requires choosing and tuning an additional regression, quantization or cubature layer—decisions that are outside the scope of their papers and would introduce a subjective bias into any comparison. Consequently, our tests focus on accuracy versus analytical solutions and on convergence-rate diagnostics, which are the only fair yardsticks at present.

\section*{Acknowledgments}
The research of Kristoffer Andersson was funded by the RIBA2022 grant.

\bibliographystyle{plain}
\bibliography{references}

\appendix
\section{Proof of Lemma 4.1}\label{app:proof_41}
In this appendix we provide the full proof of Lemma~\ref{lemma:main_res}, restated here for convenience.
\begin{lemma}
Assume that $h$ is small enough so that $(\color{black}8K_1^2 + 1\color{black})\,h<1$.
For $j \in \{1,2\}$, suppose 
  $\bigl\{\widehat{Y}_n^{k,\pi,j}, \widehat{Z}_n^{k,\pi,j}\bigr\}_{(k,n)\in \Delta^\pi[0,T]^2}$ 
  are two square integrable solutions to \eqref{eq:BSDE_family_disc_generic}. 
  Define the differences
  \begin{equation*}
    \delta Y_n^k \;=\; \widehat{Y}_n^{k,\pi,1} \;-\; \widehat{Y}_n^{k,\pi,2},
    \qquad
    \delta Z_n^k \;=\; \widehat{Z}_n^{k,\pi,1} \;-\; \widehat{Z}_n^{k,\pi,2}.
  \end{equation*}
  Then there exist constants $C_Y$ and $C_Z$, depending only on $T$ and $K_1$, such that 
  for every $(k,n)\in \Delta^\pi[0,T]^2$, the following estimates hold:
  \begin{align*}
    \E\bigl|\delta Y_n^k\bigr|^2 
    &\le
      C_Y \Bigl(\,
        \E\bigl|\delta Y_N^k\bigr|^2 
        \;+\; 
        \sum_{\ell=n}^{N-1}\E\bigl|\delta Y_N^\ell\bigr|^2\,h
      \Bigr),\quad
    \E\bigl|\delta Z_n^k\bigr|^2 \, h
    \le\ 
      C_Z \Bigl(\,
        \E\bigl|\delta Y_N^k\bigr|^2 
        \;+\; 
        \sum_{\ell=n}^{N-1}\E\bigl|\delta Y_N^\ell\bigr|^2\,h
      \Bigr).
  \end{align*}
    \begin{proof}
Let $\delta f_n^k=f(t_k,t_n,X_n^\pi,Y_n^{n,\pi,1},Z_n^{k,\pi,1}) -  f(t_k,t_n,X_n^\pi,Y_n^{n,\pi,2},Z_n^{k,\pi,2})$, then \color{black}
\[
\delta Y_{n+1}^k=\delta Y_n^k-\delta f_n^k\,h+(\delta Z_n^k)^\top\Delta W_n,
\qquad
\delta Z_n^k=\frac{1}{h}\,\E\!\big[\delta Y_{n+1}^k\Delta W_n\mid\F_{t_n}\big],
\]
so that in particular
\(
\E[\delta Y_{n+1}^k\mid\F_{t_n}]=\delta Y_n^k-\delta f_n^k h.
\) \color{black}
By the martingale representation theorem, there exists an adapted, square integrable process $(\delta Z_t^k)_{t\in[t_n,t_{n+1}]}$, such that
\begin{align*}
\delta Y_{n+1}^k= \E\big[\delta Y_{n+1}^k\,|\,\F_{t_n}\big] + \int_{t_n}^{t_{n+1}}\delta Z_t^k \d W_t=\delta Y_n^k - \delta f_n^kh + \int_{t_n}^{t_{n+1}}\delta Z_t^k \d W_t.
\end{align*}
 Since $\delta f_n^k$ and $\delta Y_n^k$ are $\F_{t_n}$-measurable, it holds that $\E\big[\delta Y_n^k\int_{t_n}^{t_{n+1}}\delta Z_t^k \d W_t\big]=\E\big[\delta f_n^k\int_{t_n}^{t_{n+1}}\delta Z_t^k \d W_t\big]\\=0$ and by Itô-Isometry, we have    
$ \E|\delta Y_{n+1}^k|^2= \E|\delta Y_n^k - \delta f_n^kh|^2 + \int_{t_n}^{t_{n+1}}\E|\delta Z_t^k|^2\d t$.
We distinguish two cases, depending on whether $n>k$ or $n=k$. The difference arises in how we estimate the term $\E\bigl|\delta Y_n^k - \delta f_n^k h\bigr|^2$.
\color{black}

\vspace{0.25cm}\noindent \textbf{Case I} ($n>k$):\newline
\color{black}
For the first term on the right-hand side, we apply Cauchy--Schwarz and Young’s
inequality with parameter $\lambda>0$, and use the Lipschitz continuity of $f$, to get
\begin{align*}
  \E\big|\delta Y_n^k - \delta f_n^k h\big|^2
  &\ge \E|\delta Y_n^k|^2 - 2h\,\E[\delta Y_n^k\,\delta f_n^k] \\
  &\ge \E|\delta Y_n^k|^2 - \lambda h\,\E|\delta Y_n^k|^2
      - \frac{2K_1^{2}}{\lambda}h\Big(\E|\delta Y_n^n|^2 + \E|\delta Z_n^k|^2\Big).
\end{align*}
\color{black}
We further observe that 
$\E\bigl[\int_{t_n}^{t_{n+1}}\delta Z_t^k\mathrm{d}t\,\big|\,\F_{t_n}\bigr] = h\delta Z_n^k$, see e.g., \cite[Lemma 1]{han2020convergence},
 which, together with the Cauchy--Schwarz inequality, implies 
$\int_{t_n}^{t_{n+1}}\E|\delta Z_t^k|^2\mathrm{d}t \geq \E|\delta Z_n^k|^2h.$ Collecting terms then yields
\color{black}
\begin{align*}
  \E|\delta Y_{n+1}^k|^2
  \;\ge\; (1-\lambda h)\,\E|\delta Y_n^k|^2
          - \frac{2K_1^{2}}{\lambda}h\,\E|\delta Y_n^n|^2
          + \Big(1-\frac{2K_1^{2}}{\lambda}\Big)h\,\E|\delta Z_n^k|^2.
\end{align*}
Choose, for instance, $\lambda=4K_1^{2}$. Then
\begin{align*}
  (1-4K_1^{2}h)\,\E|\delta Y_n^k|^2
  + \tfrac12 h\,\E|\delta Z_n^k|^2
  \;\le\; \,\E|\delta Y_{n+1}^k|^2 + \frac{1}{2}h\,\E|\delta Y_n^n|^2.
\end{align*}
For sufficiently small $h$ such that $4K_1^{2}h<1$, this implies
\begin{align}
\label{ineq:Ynk_Znk}
\begin{split}
      \E|\delta Y_n^k|^2
  \;\le\; (1-4K_1^{2}h)^{-1}\!\Big( \E|\delta Y_{n+1}^k|^2 + \tfrac12 h\,\E|\delta Y_n^n|^2 \Big),\quad h\,\E|\delta Z_n^k|^2
  \;\le\; \,2\E|\delta Y_{n+1}^k|^2 + h\,\E|\delta Y_n^n|^2.
  \end{split}
\end{align}
Setting $A_1=(1-4K_1^{2}h)^{-1}$ and iterating the first inequality above yield
\begin{align}\label{ineq:deltaY_nk}
  \E|\delta Y_n^k|^2
  \;\le\; A_1^{\,N-n}\E|\delta Y_N^k|^2
        + \frac12\sum_{\ell=n}^{N-1}A_1^{\,\ell+1-n}\,\E|\delta Y_\ell^\ell|^2\,h.
\end{align}
\color{black}
\vspace{0.25cm}\newline \textbf{Case II} ($n=k$):\newline
We again apply Cauchy--Schwarz and Young’s inequality (with parameter $\lambda>0$)
and use the Lipschitz continuity of $f$ to obtain
\color{black}
\begin{align*}
  \E|\delta Y_{n+1}^n|^2
  &\ge \E|\delta Y_n^n|^2 - 2h\,\E[\delta Y_n^n\,\delta f_n^n] + h\,\E|\delta Z_n^n|^2 \\
  &\ge \E|\delta Y_n^n|^2 - \lambda h\,\E|\delta Y_n^n|^2
      - \frac{2K_1^2}{\lambda}h\Big(\E|\delta Y_n^n|^2+\E|\delta Z_n^n|^2\Big)
      + h\,\E|\delta Z_n^n|^2 \\
  &= \Big(1-\lambda h-\tfrac{2K_1^2}{\lambda}h\Big)\E|\delta Y_n^n|^2
     + \Big(1-\tfrac{2K_1^2}{\lambda}\Big)h\,\E|\delta Z_n^n|^2.
\end{align*}
\color{black}
\color{black}
As in Case I, choose $\lambda=4K_1^2$. Then
\color{black}
\[
  \E|\delta Y_{n+1}^n|^2
  \;\ge\; \Big(1-(4K_1^2+\tfrac12)h\Big)\E|\delta Y_n^n|^2
          + \tfrac12 h\,\E|\delta Z_n^n|^2.
\]
Hence
\begin{align}\label{ineq:delta_Y_nn_Z_nn}
      \E|\delta Y_n^n|^2\le \Big(1-(4K_1^2+\tfrac12)h\Big)^{-1}\E|\delta Y_{n+1}^n|^2,
  \qquad
  h\,\E|\delta Z_n^n|^2 \le 2\,\E|\delta Y_{n+1}^n|^2.
\end{align}
\color{black}
Setting \(A_2=\Big(1-(4K_1^2+\tfrac12)h\Big)^{-1}\) and combining~\eqref{ineq:deltaY_nk} and the first inequality in~\eqref{ineq:delta_Y_nn_Z_nn} yield
\begin{align}
\label{ineq:sym}
\begin{split}
        \E|\delta Y_n^n|^2
    &
    \leq
    A_2A_1^{N-n-1}\E|\delta Y_N^n|^2 + \frac{A_2}{2}\sum_{\ell=n+1}^{N-1}A_1^{\ell-n}\E|\delta Y_\ell^\ell|^2h.
    \end{split}
\end{align}
For \(m\in\{0,1,\ldots,N-1\}\), set \(u_m=A_1^{-(N-m)}\E|\delta Y_m^m|^2\), then~\eqref{ineq:sym} can be written as
\begin{align*}
    u_n\leq A_2A_1^{-1}\E|\delta Y_N^n|^2 + \frac{A_2}{2}\sum_{\ell=n+1}^{N-1}u_\ell h.
\end{align*}
Applying the discrete Grönwall's inequality to the above yields
\begin{align*}
    u_n
    \leq 
    A_2A_1^{-1}\E|\delta Y_N^n|^2\e^{\frac{A_2}{2}(N-n)h}.
\end{align*}
Since \((N-n)h=T-t_n\) and \(\E|\delta Y_n^n|^2=u_nA_1^{N-n}\), we obtain 
\begin{align}\label{ineq:Y_nn}
    \E|\delta Y_n^n|^2
    \leq
    A_1^{N-n-1}A_2\e^{\frac{A_2}{2}(T-t_n)}\E|\delta Y_N^n|^2.
\end{align}
This completes Case II. 

Using \eqref{ineq:Y_nn} in \eqref{ineq:deltaY_nk} and noting that $A_1^{\ell-n}$ is increasing in $\ell$, we obtain
\begin{equation}\label{ineq:Y_nk}
    \E|\delta Y_n^k|^2
    \le
    A_1^{N-n}\E|\delta Y_N^k|^2
    +
    \frac{1}{2}A_2\,A_1^{N-n}\e^{\frac{A_2}{2}T}
    \sum_{\ell=n}^{N-1}\E|\delta Y_N^\ell|^2\,h.
\end{equation}
To bound $h\,\E|\delta Z_n^k|^2$, note from the second inequalities in \eqref{ineq:Ynk_Znk} and \eqref{ineq:delta_Y_nn_Z_nn} that, for all $(k,n)\in\Delta^\pi[0,T]^2$,
\[
  h\,\E|\delta Z_n^k|^2 \;\le\; 2\,\E|\delta Y_{n+1}^k|^2 \;+\; h\,\E|\delta Y_n^n|^2,
\]
(and for $k=n$ a sharper version without the last term holds). Using \eqref{ineq:deltaY_nk} at time $n+1$ and \eqref{ineq:Y_nn}, together with $\e^{\frac{A_2}{2}(T-t_\ell)}\le \e^{\frac{A_2}{2}T}$, we obtain
\begin{equation}\label{ineq:Znk_final}
  h\,\E|\delta Z_n^k|^2
  \;\le\;
  2A_1^{N-n-1}\,\E|\delta Y_N^k|^2
  \;+\;
  A_2\,\e^{\frac{A_2}{2}T}\,A_1^{N-n-1}
  \sum_{\ell=n}^{N-1}\E|\delta Y_N^\ell|^2\,h.
\end{equation}

The final step of the proof is to show that there exist constants \(C_Y, C_Z > 0\) independent of \(h\).
Take \(h\) sufficiently small, such that
$h\leq\frac{1}{8K_1^2+1}$, then $4K_1^2h\le \tfrac12$ and $(4K_1^2+\tfrac12)h\le \tfrac12$. Hence, for \(0\leq n \leq N\),
\[
A_1^n \le \e^{8K_1^2 T},\qquad
A_2\le 2,\qquad
\e^{\frac{A_2}{2}T}\le \e^{T}.
\]
Applying these bounds to \eqref{ineq:Y_nk} and \eqref{ineq:Znk_final} yields, for all $(k,n)\in\Delta^\pi[0,T]^2$,
\[
  \E|\delta Y_n^k|^2
  \le
  \e^{8K_1^2T}\,\E|\delta Y_N^k|^2
  + \e^{(8K_1^2+1)T}\,\sum_{\ell=n}^{N-1}\E|\delta Y_N^\ell|^2\,h,
\]
and
\[
  h\,\E|\delta Z_n^k|^2
  \le
  2\e^{8K_1^2T}\,\E|\delta Y_N^k|^2
  + 2\e^{(8K_1^2+1)T}\,\sum_{\ell=n}^{N-1}\E|\delta Y_N^\ell|^2\,h.
\]
Hence the lemma holds with
\[
C_Y := \e^{(8K_1^2+1)T},
\qquad
C_Z := 2\e^{(8K_1^2+1)T}.
\]
\end{proof}
\end{lemma}
\end{document}